\documentclass[a4paper,11pt]{amsart}

\usepackage[utf8]{inputenc}
\usepackage{amsmath}		
\usepackage{amssymb}		
\usepackage{enumerate}
\usepackage{amsthm}
\usepackage{mathtools}
\usepackage{bbm} 
\usepackage[titletoc,title]{appendix}
\usepackage{enumitem,linegoal}
\usepackage{calc}

\usepackage{bm}

\usepackage[style = alphabetic,firstinits=true,doi=false,isbn=false,url=false,
backend=biber]{biblatex}
\renewbibmacro{in:}{}
\renewbibmacro*{volume+number+eid}{%
  \printfield{volume}%
  \printfield{number}%
  \setunit{\addcomma\space}%
  \printfield{eid}}
\DeclareFieldFormat[article]{number}{\mkbibparens{#1}}

\usepackage{chngcntr}
\usepackage{apptools}
\AtAppendix{\counterwithin{thm}{section}}

\usepackage[hidelinks,unicode,pdftex,pdfpagelabels,bookmarks,hyperindex]{hyperref}

\newcommand\restr[2]{{
  \left.\kern-\nulldelimiterspace 
  #1 
  \vphantom{\big|} 
  \right|_{#2} 
  }}
  
\newcommand\srestr[2]{{
  \left.\kern-\nulldelimiterspace 
  #1 
  \right|_{#2} 
  }}

\newcommand{\calH}{\mathcal{H}}
\newcommand{\calE}{\mathcal{E}}

\newcommand{\RR}{\mathbb{R}}

\newcommand{\NN}{\mathbb{N}}
\newcommand{\Sph}{\mathbb{S}}

\newcommand*\intdiff{\mathop{}\!\mathrm{d}}

\newcommand{\abs}[1]{\lvert#1\rvert}
\newcommand{\norm}[1]{\lVert#1\rVert}

\newcommand{\dmu}{\intdiff \mu}
\newcommand{\dmui}{\intdiff \mu_i}

\newcommand{\nui}{\nu_i}

\newcommand{\ve}{V^\epsilon}
\newcommand{\vei}{V^i}
\newcommand{\vi}{V^i}
\newcommand{\veip}{V^{i'}}

\newcommand{\nv}{\norm{V}}
\newcommand{\nvi}{\norm{V^i}}
\newcommand{\nve}{\norm{V^\eps}}

\newcommand{\dimh}{\dim_\calH}

\newcommand{\dv}{\intdiff V}
\newcommand{\dve}{\intdiff \ve}
\newcommand{\dvei}{\intdiff \vei}

\newcommand{\dveip}{\intdiff \veip}
\newcommand{\dvi}{\intdiff \vi}

\newcommand{\dnve}{\intdiff \nve}

\newcommand{\dnvi}{\intdiff \nvi}

\newcommand{\dnv}{\intdiff \nv}

\newcommand{\Ae}{A^{\epsilon}}

\newcommand{\Aei}{A^{i}}
\newcommand{\Bei}{B^{\epsilon_i}}

\newcommand{\eps}{\epsilon}
\newcommand{\dH}{\intdiff \calH}
\newcommand{\pd}{\partial}
\renewcommand{\div}{\mathrm{div}}
\newcommand{\linspan}{\mathrm{span}}

\DeclareMathOperator\ind{ind}

\DeclareMathOperator\spt{spt}
\DeclareMathOperator\reg{reg}
\DeclareMathOperator\sing{sing}
\DeclareMathOperator\Ric{Ric}
\DeclareMathOperator{\dist}{dist}

\DeclareMathOperator{\inj}{inj}

\DeclareMathOperator{\grad}{grad}

\theoremstyle{plain}
\newtheorem{thm}{Theorem}[section]
\newtheorem*{thm*}{Theorem}
\newtheorem{lem}[thm]{Lemma}
\newtheorem{prop}[thm]{Proposition}
\newtheorem{cor}[thm]{Corollary}
\newtheorem*{cor*}{Corollary}

\theoremstyle{definition}
\newtheorem{defn}[thm]{Definition}

\newtheorem*{ackn}{Acknowledgments}
\newtheorem*{outline}{Outline of the paper}

\theoremstyle{remark}
\newtheorem{rem}[thm]{Remark}
\newtheorem*{rem*}{Remark}

\newtheorem{claim}{Claim}
\newtheorem*{claim*}{Claim}
\newtheorem*{notation*}{Notation}

\theoremstyle{plain}
\newtheorem{innercustomthm}{Theorem}
\newenvironment{customthm}[1]
  {\renewcommand\theinnercustomthm{#1}\innercustomthm}
  {\endinnercustomthm}
\newtheorem{innercustomcor}{Corollary}
\newenvironment{customcor}[1]
  {\renewcommand\theinnercustomcor{#1}\innercustomcor}
  {\endinnercustomcor} 

\numberwithin{equation}{section}

\begin{document}

\title[Spectrum and index of Allen--Cahn minimal hypersurfaces]
{Spectrum and index of two-sided\\ Allen--Cahn minimal hypersurfaces}
\author[Fritz Hiesmayr]{Fritz Hiesmayr}
\address{Centre for Mathematical Sciences
\\University of Cambridge
\\Cambridge CB3 0WB, United Kingdom}
\email{\url{f.l.hiesmayr@maths.cam.ac.uk}}
\date{\today}

\begin{abstract}
The combined work of Guaraco, Hutchinson, Tonegawa and Wickramasekera 
has recently  produced a new proof of the classical theorem 
that any closed Riemannian manifold of dimension $n + 1 \geq 3$ contains a minimal 
hypersurface with a singular set of Hausdorff dimension at most $n-7$.
This proof avoids the Almgren--Pitts geometric min-max procedure for the area
functional that was instrumental in the original proof, and is instead based on a 
considerably simpler PDE min-max construction of critical points
of the Allen--Cahn functional.
Here we prove a spectral lower bound the hypersurfaces that arise
from sequences of critical points with bounded indices.
%
%
In particular, the index of two-sided minimal hypersurfaces constructed
using multi-parameter Allen--Cahn min-max methods is bounded above
by the number of parameters used in the construction.
Finally, we point out by an elementary inductive argument how the
regularity of the hypersurface follows from the corresponding result
in the stable case.

\end{abstract}

\maketitle

\section{Introduction}

A classical theorem, due to the combined work of Almgren, Pitts and Schoen--Simon,
asserts that for $n \geq 2$, every $(n+1)$-dimensional closed Riemannian manifold $M$ contains a
minimal hypersurface smoothly embedded away from a closed singular set of Hausdorff
dimension at most $n-7$. The original proof of this theorem is based on a highly
non-trivial geometric min-max construction due to Pitts~\cite{Pitts81},
which extended earlier work of Almgren~\cite{Almgren65}.
This construction is carried out directly for
the area functional on the space of hypersurfaces equipped with an appropriate weak
topology, and it yields in the first instance a critical point of area satisfying
a certain almost-minimizing property.
This property is central to the rest of the argument, and allows
to deduce regularity of the min-max hypersurface from compactness of
the space of uniformly area-bounded stable minimal hypersurfaces
with singular sets of dimension at most $n-7$,
a result proved for $2 \leq n \leq 5$ by Schoen--Simon--Yau~\cite{SchoenSimonYau75} and
extended to arbitrary $n \geq 2$ by Schoen--Simon~\cite{SchoenSimon81}.
(The Almgren--Pitts min-max construction has recently been streamlined by
De Lellis and Tasnady~\cite{DeLellisTasnady13} giving a shorter proof.
However, their argument still follows Pitts' closely and is in
particular based on carrying out the min-max procedure directly
for the area functional on hypersurfaces.)

In recent years an alternative approach to this theorem has been developed, whose philosophy is to push the regularity theory to its limit in order to gain substantial simplicity on the existence part.
Specifically, this approach differs from the original one in two key aspects: 
first, it is based on a strictly PDE-theoretic min-max construction that replaces the Almgren--Pitts geometric construction;
second, for the regularity conclusions, it relies on a sharpening of
the Schoen--Simon compactness theory for stable minimal hypersurfaces.
The idea in this approach is to construct a minimal hypersurface as the
limit-interface associated with a sequence of solutions $u = u_{i}$ to the Allen--Cahn equation
\begin{equation}
\label{eq:allen_cahn_equation}
\Delta u -\epsilon_{i}^{-2} W'(u) = 0
\end{equation}
on the ambient space $M$, where $W \! : {\mathbb R} \to {\mathbb R}$
is a fixed double-well potential with precisely two minima at $\pm 1$ with $W(\pm1) = 0$.
Roughly speaking, if the $u_{i}$ solve \eqref{eq:allen_cahn_equation} and satisfy appropriate bounds, then the level sets of $u_i$
converge as $\epsilon_{i} \to 0^{+}$ to a stationary codimension
$1$ integral varifold $V$.
This fact was rigorously established by
Hutchinson--Tonegawa~\cite{HutchinsonTonegawa00}, using in part methods inspired
by the earlier work of Ilmanen in the parabolic setting Ilmanen~\cite{Ilmanen93}.
Note that $u_{i}$ solves \eqref{eq:allen_cahn_equation} if and only if
it is a critical point of the Allen--Cahn functional
\begin{equation}
E_{\epsilon_{i}} (u)
= \int_U \epsilon_{i} \frac{\abs{\nabla u}^{2}}{2} + \frac{W(u_i)}{\eps_i}.
\end{equation}
If the solutions $u_{i}$ are additionally assumed stable with respect to
$E_{\epsilon_{i}}$, then Tonegawa and Wickramasekera \cite{TonegawaWickramasekera10}
proved that the resulting
varifold $V$ is supported on a hypersurface smoothly embedded away from a closed
singular set of Hausdorff dimension at most $n-7$, using an earlier result of
Tonegawa \cite{Tonegawa05} which established the stability of the regular part
$\reg V$ with respect to the area functional.
Their proof of this regularity result uses the regularity and compactness theory
for stable codimension 1 integral varifolds developed by
Wickramasekera~\cite{Wickramasekera14} sharpening the Schoen--Simon theory.

Stability of $u_i$ means that the second variation of the Allen--Cahn functional
$E_{\eps_i}$ with respect to $H^1(M)$ is a non-negative quadratic form.
More generally the index $\ind u_i$ denotes the number of strictly negative
eigenvalues of the elliptic operator $L_i = \Delta - \eps_i^{-2}W''(u_i)$,
so that $u_i$ is stable if and only if $\ind u_i = 0$.
Using min-max methods for semi-linear equations,
Guaraco~\cite{Guaraco15} recently gave a simple and elegant construction
of a solution $u_{i}$ to \eqref{eq:allen_cahn_equation}
with $\ind  u_{i} \leq 1$ and $\norm{u_i}_{L^\infty} \leq~1$,
and such that as $\epsilon_{i} \to 0$, the energies
$E_{\epsilon_{i}}(u_{i})$ are bounded above and below away from 0.
The lower energy bound
guarantees that the resulting limit varifold $V$ is non-trivial.
Since $\ind u_i \leq 1$, $u_i$ must be stable in at least one of
every pair of disjoint open subsets of $M$; similarly if
$\ind u_i \leq k$ then $u_i$ must be stable in at least one of every
$(k+1)$-tuple of disjoint open sets. This elementary observation,
originally due to Pitts in the context of minimal surfaces, together
with a tangent cone analysis in low dimensions, allowed Guaraco
to deduce the regularity of $V$ from the results of \cite{TonegawaWickramasekera10}.
More recently still, Gaspar and Guaraco~\cite{GasparGuaraco2016}
have used $k$-parameter min-max methods to produce sequences of
critical points with Morse index at most $k$, for all positive
integers~$k$.
Our results show that this index bound is inherited by
the minimal surface arising as $\eps_i \to 0$, provided
it has a trivial normal bundle.
We also point out that the regularity follows
in all dimensions from the corresponding result in the stable case via
an inductive argument that avoids the tangent cone analysis used in~\cite{Guaraco15}.

\begin{cor*}
Let $M$ be a closed Riemannian manifold of dimension $n+1 \geq 3$.
Let $V$ be the integral varifold arising as the limit-interface of
the sequence $(u_i)$ of solutions to \eqref{eq:allen_cahn_equation}
constructed in~\textup{\cite{Guaraco15}}
(respectively in~\textup{\cite{GasparGuaraco2016}} using $k$-parameter min-max	methods).
Then $\dimh \sing V \leq n-7$. If $\reg V$ is two-sided,
then its Morse index with respect to the
area functional satisfies $\ind_{\calH^n} \reg V \leq 1$
(respectively $\ind_{\calH^n} \reg V \leq k$).
\end{cor*}
In min-max theory, one generally expects that the Morse index of the constructed critical point is no greater than the number of parameters
used in the construction.
The above corollary gives this result for the constructions of Guaraco and
Gaspar--Guaraco, provided the arising hypersurface is two-sided.
This was recently shown by Chodosh and Mantoulidis~\cite{ChodoshMantoulidis2018}
to hold automatically when the ambient manifold $M$
has dimension $3$ and is equipped either with a bumpy metric or has positive
Ricci curvature. Building on work of Wang and Wei~\cite{WangWei2017}, Chodosh--Mantoulidis
prove curvature and strong sheet separation estimates, and use these
to deduce that in this three-dimensional setting the convergence of the level
sets occurs with multiplicity $1$.
They moreover show that in all dimensions, if the limiting surface has
multiplicity $1$, then its index is bounded \emph{below} by the index
of the $u_\eps$.

This complements our upper bound for the index, which is a direct consequence
of a lower bound for $(\lambda_p)$, the spectrum of the elliptic
operator $L_V = \Delta_V + \abs{A}^2 + \Ric_M(\nu,\nu)$---the \emph{scalar
Jacobi operator}---in terms of $(\lambda_p^i)$, the spectra of the operators $(L_i)$.
Establishing this spectral lower bound is our main result.
\begin{thm*}
Let $M$ be a closed Riemannian manifold of dimension $n+1 \geq 3$.
Let $V$ be the integral
varifold arising from a sequence $(u_i)$ of solutions to
\eqref{eq:allen_cahn_equation} with $\ind u_i \leq k$ for some $k \in \NN$.
Then $\dimh \sing V \leq n-7$ and
\begin{enumerate}[font = \upshape, label = (\roman*)]
	\item $\lambda_p(W) \geq \limsup_{i \to \infty} \lambda_p^i(W)$ for all
	$W \subset \subset M \setminus \sing V$ and $p \in \NN$, \label{item:intro_thm_spec_lower_bd}
	\item $\ind_{\calH^n} C \leq k$ for every two-sided connected component
	$C \subset \reg V$.
\end{enumerate}
\end{thm*}
\begin{rem*}
The spectral lower bound of~\ref{item:intro_thm_spec_lower_bd} also holds
if the assumptions on the $u_i$ are weakened in a spirit similar to the	work of
Ambrozio, Carlotto and Sharp~\cite{AmbrozioCarlottoSharp2015},
that is if instead of an index upper bound one assumes that for some
$k \in \NN$ there is $\mu \in \RR$ such that $\lambda_k^i \geq \mu$ for all $i$.
(Note that the index bound $\ind u_i \leq k$ is equivalent to
$\lambda_{k+1}^i \geq 0$.)
\end{rem*}
\begin{rem*}
It was recently brought to our attention that a similar result
had previously been proved by Le~\cite{Le2011} in ambient Euclidean space,
under the additional assumption that the convergence to the
limit surface occurs with multiplicity $1$.
Adapting the methods developed in~\cite{Le2011,Le2015}
to ambient Riemannian manifolds, Gaspar generalised our
results to the case where the limit varifold is one-sided,
without any assumption on multiplicity~\cite{Gaspar2017}.
Their general approach is similar to ours but subtly different,
in that they instead consider the second \emph{inner}
variation of the Allen--Cahn functional; see also the recent
work of Le and Sternberg~\cite{LeSternberg2018}, where similar bounds are established
for other examples of eigenvalue problems.
\end{rem*}

For the minimal hypersurfaces obtained by a direct min-max
procedure for the area functional on the space of hypersurfaces (as in the Almgren--Pitts construction), index bounds have recently been established
by Marques and Neves~\cite{MarquesNeves15}.
Both the Almgren--Pitts existence proof and the Marques--Neves proof of
the index bounds are rather technically involved; in particular, the
min-max construction in this setting has to be carried out in a bare-handed
fashion in the absence of anything like a Hilbert space structure.
In contrast, in the approach via the Allen--Cahn functional, Guaraco's
existence proof is strikingly simple, and our proofs for the spectral
bound and the regularity of $V$ are elementary bar the fact that
they rely on the highly non-trivial sharpening of the
Schoen--Simon regularity theory for stable hypersurfaces as
in~\cite{Wickramasekera14}.

\begin{outline}
In Section~\ref{sec:main_statement} we briefly expose notions from the theory of
varifolds, set the context for the rest of the paper and give the statements
of the main result and its corollaries.
Their proof requires a number of preliminary results, which are contained in
Section~\ref{section:preliminary_results}.
The proof of the main result (Theorem~\ref{thm:main_thm_propagation_of_index_bounds})
is in Section~\ref{sec:proof_main_thm}, and is split into two parts:
in the first part we prove the spectral lower bound by an inductive
argument on $\ind u_i$; this immediately implies the
index upper bound. The proof of $\dimh \sing V \leq n-7$ is given in the second
part, and uses a similar inductive argument.
There are two appendices: Appendix~\ref{app:measure_function_convergence} contains
two elementary lemmas from measure theory that are used repeatedly in
Section~\ref{section:preliminary_results}.
Appendix~\ref{app:second_fundamental_form_coordinate_expression}
gives a proof of Proposition~\ref{prop:weak_convergence_second_ff},
which is a straight-forward adaptation of an argument used by
Tonegawa for the stable case.
\end{outline}

\begin{ackn}
I would like to thank my PhD supervisor Neshan Wickramasekera for his
encouragement and support, and Otis Chodosh for helpful conversations.
This work was supported by the UK Engineering and Physical Sciences Research Council
(EPSRC) grant EP/L016516/1 for the University of Cambridge Centre for Doctoral 
Training, the Cambridge Centre for Analysis.
\end{ackn}

\section{Varifolds, stability and statement of main theorem}
\label{sec:main_statement}
The setting is as follows: 
$(M^{n+1},g)$ is a closed (that is, compact without boundary) Riemannian manifold
of dimension $n+1 \geq 3$, and $U \subset M$ is an arbitrary
open subset, possibly equal to $M$ itself.

\subsection{Varifolds: basic definitions}
An $n$-dimensional \emph{varifold} in $U$ is a Radon measure
on the Grassmannian $G_n(U) = \{(p,E) \mid p \in U, E \subset T_pM,
\dim E = n\}$---the space of $n$-dimensional planes over points in $U$.

An important subclass are the \emph{integral varifolds},
which correspond to a pair $(\Sigma,\theta)$ of a countably $n$-rectifiable set
$\Sigma \subset U$ and a Borel-measurable function $\theta \in L_{\mathrm{loc}}^1(\Sigma,\NN)$ via
\begin{equation}
V_{\Sigma,\theta}(\phi) = \int_U \phi(x,T_x\Sigma) \theta(x) \dH^n(x)
\quad \text{for all $\phi \in C_c(G_n(U))$},
\end{equation}
where $T_x\Sigma$ is the $\calH^n$-a.e. defined tangent space to $\Sigma$.
The function $\theta$ is called the \emph{multiplicity function}.

A sequence $(\vi)$ \emph{converges as varifolds} to $V$
if they converge weakly as Radon measures on $G_n(U)$, i.e.\ if
\begin{equation}
\int_{G_n(U)} \phi \dvi \to \int_{G_n(U)} \phi \dv
\quad
\text{for all $\phi \in C_c(G_n(U))$.}
\end{equation}

The \emph{weight measure} $\norm{V}$ of a varifold $V$ is defined by
\begin{equation}
	\norm{V}(\phi) = \int_{G_n(U)} \phi(x) \dv(x,S)
\quad
\text{for all $\phi \in C_c(U)$}.
\end{equation}

Consider an arbitrary vector field $X \in C_c^1(U,TM)$ with
flow $(\Phi_t)$. We deform $V$ in the direction
of $X$ by pushing it forward via its flow, that is
\begin{equation}
(\Phi_t)_*V(\phi) = \int_{G_n(U)} \phi(\Phi_t(x),\intdiff \Phi_t(x) \cdot S) J \Phi_t(x) \dv(x,S)
\end{equation}
for all $\phi \in C_c(G_n(U))$, 
where $J\Phi_t(x) =
\det (\mathrm{d} \Phi_t(x)^* \circ \mathrm{d} \Phi_t(x)) ^{\frac{1}{2}}$
is the Jacobian of $\Phi_t(x)$.

Differentiating the corresponding weight measures
$\lVert (\Phi_t)_* V \rVert$ yields the \emph{first variation} of $V$:
\begin{equation}
\delta V(X) = \srestr{\frac{\mathrm{d}}{\mathrm{d}t}}{t = 0} \norm{(\Phi_t)_*V}(U).
\end{equation}
When $\delta V(X) = 0$ for all vector fields $X \in C_c^1(U,TM)$, we say
that $V$ is \emph{stationary} in $U$.

By definition, the \emph{regular part} of $V$ is the set of points 
$x \in U \cap \spt \lVert V \rVert$ such that in a neighbourhood of $x$,
$\spt \lVert V \rVert$ is smoothly embedded in $M$.
Its complement is the \emph{singular part} of $V$,
denoted $\sing V := U \cap \spt \nv \setminus \reg V$. 
For a stationary integral varifold $V$, the Allard regularity theorem
implies that $\reg V$ is a dense subset of $U \cap \spt \lVert V \rVert$
\cite[Ch.~5]{Simon84}.
\subsection{Stability and the scalar Jacobi operator}
Throughout this section $V$ will be a stationary integral $n$-varifold in $U \subset M$.
We call $V$ \emph{two-sided} if its regular part $\reg V$
is two-sided, that is if the normal bundle $NV := N(\reg V)$
admits a continuous non-vanishing section.
When this fails, $V$ is called \emph{one-sided}.
(Recall that when the ambient manifold $M$ is orientable,
then $\reg V$ is two-sided if and only if it is orientable.)

Suppose that $V$ is two-sided, and fix a unit normal vector field
$N \in C^1(NV)$, so that every function $\phi \in C_c^1(\reg V)$
corresponds to a section $\phi N \in C_c^1(NV)$ and vice-versa.
After extending the vector field $\phi N$ to $C_c^1(U,TM)$---the chosen
extension will not matter for our purposes---we can deform $\reg V$
with respect to its flow $(\Phi_t)$.
As $V$ is stationary, the first variation vanishes: $\delta V(\phi N) = 0$.
A routine calculation, the details of which can be found for instance
in~\cite[Ch.~2]{Simon84} shows that the second variation satisfies
\begin{multline}
\label{eq:second_variation_smooth}
\delta^2 V(\phi N) =
\srestr{\frac{\intdiff^2}{\intdiff t^2}}{t = 0} \norm{(\Phi_t)_* V}(U) 
= \\ \int_U \abs{\nabla_V \phi}^2 -
( \abs{A}^2 + \Ric_M(N,N) )\phi^2 \dnv,
\end{multline}
where $\nabla_V$ is the Levi-Civita connection on $\reg V$,
$A$ is the second fundamental form of $\reg V \subset M$, and 
$\Ric_M$ is the Ricci curvature tensor on $M$.

The expression on the right-hand side can be defined for one-sided
$V$ by replacing $N$ by an arbitrary measurable unit section
$\nu: \reg V \to NV$, but it loses its interpretation in terms
of the second variation of the area.

\begin{defn}[Scalar second variation]
The \emph{scalar second variation} of a stationary integral varifold $V$
is the quadratic form $B_V$ defined for $\phi \in C_c^2(\reg V)$ by
\begin{equation}
\label{defn:scalar_second_variation}
B_V(\phi,\phi) 
= \int_{\reg V} \abs{\nabla_V \phi}^2
- (\abs{A}^2 + \Ric_M(\nu,\nu)) \phi^2 \dnv.
\end{equation}

\end{defn} 
\begin{rem}
When $V$ is one-sided, the second variation of its area has to be measured
with respect to variations in $C_c^1(NV)$---we refer
to \cite[Ch.~2]{Simon84} or~\cite[Sec.~1.8]{ColdingMinicozzi11}
for further information on this.
We called $B_V$ `scalar' in order to highlight its difference with
the second variation of area in this case,
but emphasise that for the remainder `second variation' refers exclusively to the
quadratic form $B_V$ from Definition~\ref{defn:scalar_second_variation}. 
(For the same reasons we also call the Jacobi operator $L_V$
`scalar' in Definition~\ref{defn:scalar_jac} below,
but omit this adjective in the remainder of the text.)
\end{rem}
One can consider $\reg V$ as a stationary integral varifold
in its own right by identifying it with the corresponding varifold
with constant multiplicity $1$. Its scalar second variation
\begin{equation}
B_{\reg V}(\phi,\phi) = \int_{\reg V} \abs{\nabla_V \phi}^2
- (\abs{A}^2 + \Ric_M(\nu,\nu)) \phi^2 \dH^n
\end{equation}
differs from $B_V$ only in that the integral is with respect
to the $n$-dimensional Hausdorff measure instead of $\nv$. This means exactly that 
while $B_V$ is `weighted' by the multiplicity of $V$, the quadratic
form $B_{\reg V}$ measures the variation of `unweighted' area;
we will briefly use this in Section~\ref{sec:prelim_spectrum_sec_var}.

After integrating by parts on $\reg V$, the form $B_V$
corresponds to the second-order elliptic operator
$L_V = \Delta_V + \abs{A}^2 + \Ric_M(\nu,\nu),$
where $\Delta_V$ is the Laplacian on $\reg V$.
\begin{defn}[Scalar Jacobi operator]
	\label{defn:scalar_jac}
	The \emph{scalar Jacobi operator} of $V$, denoted $L_V$, is the second-order
	elliptic operator
	\begin{equation}
	L_V \phi = \Delta_V \phi + 
	(\abs{A}^2 + \Ric_M(\nu,\nu)) \phi
	\quad \text{for all $\phi \in C^2(\reg V)$},
	\end{equation}
	where $\nu: \reg V \to NV$ is an arbitrary measurable unit normal
	vector field.
\end{defn}
The curvature of $\reg V$ can blow up as one approaches $\sing V$,
in which case the coefficients of the operator $L_V$ would
not be bounded. To avoid this, we restrict ourselves to a compactly
contained open subset $W \subset \subset U \setminus \sing V$;
moreover we require
$W \cap \reg V \neq \emptyset$ to avoid vacuous statements.

We use the sign convention for the spectrum defined in
\cite[Ch.~8]{GilbargTrudinger98}, where $\lambda \in \RR$
is an eigenvalue of $L_V$ in $W$ if there is $\varphi
\in H_0^1(W \cap \reg V)$ such that $L_V \varphi + \lambda \varphi = 0$.
By standard elliptic PDE theory the spectrum
\begin{equation}
\lambda_1 \leq \lambda_2 \leq \cdots \to +\infty
\end{equation}
of $L_V$ in $W$ is discrete and bounded below.
We will sometimes also write $\lambda_p(W)$ instead of $\lambda_p$
in order to highlight the dependence of the spectrum on the
subset $W$. 
The eigenvectors of $L_V$
span the space $H_0^1(W \cap \reg V) = W_0^{1,2}(W \cap \reg V)$,
which we abbreviate throughout by $H_0^1$. 

The \emph{index of $B_V$} in $W$ is the
maximal dimension of a subspace of $H_0^1$ on which
$B_V$ is negative definite; equivalently
\begin{equation}
\ind_W B_V = \mathrm{card} \{ p \in \NN \mid \lambda_p(W) < 0 \}.
\end{equation}
Moreover $\ind B_V := \sup_W (\ind_W B_V)$, where
the supremum is taken over all $W \subset \subset U \setminus \sing V$ with
$W \cap \reg V \neq \emptyset$.
\begin{rem}
	We will see in Section~\ref{sec:prelim_spectrum_sec_var}
	that the index of $B_V$ coincides with the Morse index
	of $\reg V$ with respect to the area functional,
	at least when $\reg V$ is two-sided.
\end{rem}

\subsection{Statement of main theorem}

Let $(\eps_i)$ be a sequence of positive parameters with $\eps_i \to 0$
and consider an associated sequence of functions $(u_i)$ in $C^3(U)$ satisfying the
following hypotheses:

\begin{enumerate}[label=(\Alph*)]
\item Every $u_i \in C^3(U)$ is a critical point of the Allen--Cahn functional
\begin{equation}
\label{eq:allen_cahn_functional}
E_{\epsilon_i}(u)
= \int_U
\epsilon_i \frac{\lvert \nabla u \rvert^2}{2} + \frac{W(u)}{\epsilon_i}
\dH^{n+1},
\end{equation}
i.e. $u_i$ satisfies the equation
\begin{equation}
\label{eq:epsilon_allen_cahn_equation} 
- \epsilon_i^2 \Delta u_i + W'(u_i) = 0 \quad \text{in } U.
\end{equation} \label{hypA}

\item There exist constants $C, E_0 < \infty$ such that
\begin{equation}
\sup_i \lVert u_i \rVert_{L^\infty(U)} \leq C\; \text{and} \; \sup_i E_{\epsilon_i}(u_i) \leq E_0.
\end{equation} \label{hypB}

\item There exists an integer $k \geq 0$ such that the Morse index of each $u_i$
is at most $k$, i.e. any subspace of $C_c^1(U)$ on which the second variation
\begin{equation}
\delta^2 E_{\epsilon_i}(u_i)(\phi,\phi)
= \int_U \epsilon_i \lvert \nabla \phi \rvert^2
+ \frac{W''(u_i)}{\epsilon_i} \phi^2
\dH^{n+1}
\end{equation} 
is negative definite has dimension at most $k$. We write this $\ind u_i \leq k$,
and if $k = 0$, say that $u_i$ is \emph{stable in $U$}.
\label{hyp:C_index_bound}
\end{enumerate}
\begin{rem}
\label{rem:defn_stability}
More generally  $\ind_{U'} u_i$ denotes
the index of $\delta^2 E_{\eps_i}(u_i)$ with respect to variations in
$C_c^1(U')$ (or equivalently in $H_0^1(U')$) for all open
subsets $U' \subset U$.
When $\ind_{U'} u_i = 0$, we say that $u_i$ is \emph{stable in U'}.
\end{rem}

We follow Tonegawa~\cite{Tonegawa05}, using an idea originally
developed by Ilmanen~\cite{Ilmanen93} in a parabolic setting
to `average the level sets' of $u_i \in C^3(U)$
and define a varifold $\vei$ by
\begin{equation}
\label{eq:definition_varifolds}
\vei(\phi)
= \frac{1}{\sigma}
\int_{U \cap \{\nabla u_i \neq 0\}}
\epsilon_i \frac{\lvert \nabla u_i(x) \rvert^2}{2}
\phi(x,T_x \{u_i = u_i(x)\})
\dH^{n+1}(x)
\end{equation}
for all $\phi \in C_c(G_n(U))$.
Here $T_x \{u_i = u_i(x)\}$ is the tangent space to the level set
$\{u_i = u_i(x)\}$ at $x \in U$,
and $\sigma = \int_{-1}^1 \sqrt{W(s)/2} \intdiff s$ is a constant.
\begin{rem}
In~\cite{HutchinsonTonegawa00,Guaraco15} the varifold $\vei$ is defined by
the slightly different expression
$	\vei(\phi)
= \frac{1}{\sigma}
\int_{U \cap \{\nabla u_i \neq 0\}}
\abs{\nabla w_i(x)}  \phi(x,T_x \{u_i = u_i(x)\}) \dH^{n+1}(x)$, with
$w_i$ as in Theorem~\ref{lem:properties_V}.
The `equipartition of energy' \eqref{eq:thm_properties_V_equipartition_energy}
from Theorem~\ref{lem:properties_V}
shows that the two definitions give rise to the same limit varifold $V$
as $i \to \infty$.
\end{rem}
The weight measures $\norm{\vei}$ of these varifolds satisfy
\begin{equation}
\label{eq:density_weight_measure}
\lVert \vei \rVert(A)
= \frac{1}{\sigma} \int_{A \cap \{\nabla u_i \neq 0\}}
\epsilon_i \frac{\lvert \nabla u_i \rvert^2}{2} \dH^{n+1}
\leq \frac{E_0}{2\sigma}
\end{equation}
for all Borel subsets $A \subset U$, where the inequality follows
from the energy bound in Hypothesis~\ref{hypB}.
The resulting bound $\vei(G_n(U)) \leq \frac{E_0}{2\sigma}$ allows us to extract a
subsequence that converges to a varifold $V$, with properties laid out in the following theorem by Hutchinson--Tonegawa~\cite{HutchinsonTonegawa00}.
\begin{thm}[\cite{HutchinsonTonegawa00}]
	\label{lem:properties_V}
	Let $(u_i)$ be a sequence in $C^3(U)$ satisfying Hypotheses (A) and (B).
	Passing to a subsequence $\vei \rightharpoonup V$ as varifolds, and
	\begin{enumerate}[font = \upshape, label = (\alph*)]
		\item $V$ is a stationary integral varifold,
		\item $\nv(U) = \liminf_{i \to \infty} \frac{1}{2 \sigma} E_{\eps_i}(u_i),$
		\item for all $\phi \in C_c(U)$:
		\begin{equation}
		\label{eq:thm_properties_V_equipartition_energy}
		\lim_{i \to \infty}	\int_U
		\epsilon_i \frac{\abs{\nabla u_i}^2}{2} \phi
		= 	\lim_{i \to \infty} \int_U \frac{W(u_i)}{\epsilon_i} \phi
		= 	\lim_{i \to \infty} \int_U \abs{\nabla w_i} \phi,
		\end{equation}
		where $w_i := \Psi \circ u_i$ and
		$\Psi(t) := \int_{0}^{t} \sqrt{W(s)/2} \intdiff s$.
		
	\end{enumerate}
\end{thm}

Up to a factor of $\eps_i$ the second variation $\delta^2 E_{\eps_i}$
corresponds to the second-order elliptic operator
$L_i := \Delta - \eps_i^{-2} W''(u_i)$. As in the discussion for
the Jacobi operator, $L_i$  
has discrete spectrum $\lambda_1^i \leq \lambda_2^i \leq \cdots \to +\infty$,
which we denote by $\lambda_p^i(W)$ when we want to
emphasise its dependence on the subset $W$.
The following theorem is our main result.
\begin{customthm}{A}
\label{thm:main_thm_propagation_of_index_bounds}
Let $M^{n+1}$ be a closed Riemannian manifold, and $U \subset M$ an open subset.
Let $(u_i)$ be a sequence in $C^3(U)$ satisfying Hypotheses \ref{hypA},
\ref{hypB} and \ref{hyp:C_index_bound}, and $\vei \rightharpoonup V$.
Then $\dimh \sing V \leq n-7$ and 
\begin{enumerate}[font = \upshape, label = (\roman*)]	
	\item $\lambda_p(W) \geq \limsup_{i \to \infty} \lambda_p^i(W)$
	for all open $W \subset \subset U \setminus \sing V$ with
	$W \cap \reg V \neq \emptyset$ and all $p \in \NN$,
	\label{item:main_thm_spec}
	\item $\ind B_V \leq k$. \label{item:main_thm_ind_reg}
\end{enumerate}
\end{customthm}
\begin{rem}
The spectral lower bound remains true
if the assumptions are weakened and one assumes that for some $k \in \NN$
there is $\mu \in \RR$ such that
\begin{equation}
\lambda_k^i(U) \geq \mu \quad \text{for all $i \in \NN$}
\end{equation}
instead of an index bound---this observation is inspired
the work of Ambrozio--Carlotto--Sharp~\cite{AmbrozioCarlottoSharp2015}, where
a similar generalisation is made in the context of minimal surfaces.
One obtains the spectral bound via an inductive argument on $k$ similar to the argument
in Section~\ref{sec:proof_main_thm}, noting for the base case of
the induction that
bounds as in Corollary~\ref{cor:L2_bound_for_Ae_and_Be_if_u_stable}
hold if $\lambda_1^i \geq \mu$. 
\end{rem}
The following corollary is an immediate consequence of
Theorem~\ref{thm:main_thm_propagation_of_index_bounds}.
\begin{customcor}{B}
\label{cor:main_thm_morse}
If $\reg V$ is two-sided, then its Morse index with respect to the
area functional satisfies $\ind_{\calH^n} \reg V \leq k$.
\end{customcor}
If $V$ is the stationary varifold arising from Guaraco's $1$-parameter
min-max construction~\cite{Guaraco15} (resp.\ from the $k$-parameter min-max construction of Gaspar--Guaraco~\cite{GasparGuaraco2016})
and its regular part is two-sided, then by Corollary~\ref{cor:main_thm_morse}
its Morse index is at most $1$ (resp.\ at most $k$).

\section{Preliminary results}
\label{section:preliminary_results}
The preliminary results are divided into three parts. In the first,
following~\cite{Tonegawa05}
we introduce `second fundamental forms' $A^i$ for the varifolds $V^i$
and relate them to the second variation of the Allen--Cahn functional.
The last two sections are dedicated to the spectra of the operators
$L_V = \Delta_V + \abs{A}^2 + \Ric_M(\nu,\nu)$ and $L_i = \Delta - \eps_i^{-2} W''(u_i)$.

\subsection{Stability and \texorpdfstring{$L^2$}{L2}--bounds of curvature}
\label{subsec:generalised_second_ff}
To simplify the discussion fix for the moment a
critical point $u \in C^3(U) \cap L^\infty(U)$ of the Allen--Cahn functional
$E_\eps$, with associated varifold $V^\eps$
defined by \eqref{eq:definition_varifolds}.

Let $x \in U$ be a regular point of $u$, that is $\nabla u(x) \neq 0$.
In a small enough neighbourhood 
of $x$, the level set $\{ u = u(x)\}$ is embedded in $M$.
Call $\Sigma \subset M$ this embedded portion of the
level set, and let $A^\Sigma$ be its second fundamental form.
We use this to define a `second fundamental form' for $V^\eps$.
\begin{defn}
\label{defn:second_ff}
The function $A^\eps$ is defined at all $x \in U$ where $\nabla u(x) \neq 0$
by $A^\eps(x) = A^\Sigma(x)$.
\end{defn}
\begin{rem}
Second fundamental forms can be generalised to the context
of varifolds via the integral identity \eqref{eq:gen_second_ff}---see
Appendix~\ref{app:second_fundamental_form_coordinate_expression}, or
\cite{Hutchinson86} for the original account of this theory.
Strictly speaking it is an abuse of language to call
$A^\eps$ the `second fundamental form' of
$V^\eps$, as it satisfies this identity only up to a small error term
\eqref{eq:approximate_identity_second_fundamental_form}.
\end{rem}
By definition $\nabla_X Y = \nabla^\Sigma_X Y + A^\Sigma(X,Y)$ for all
$X,Y \in C^1(T\Sigma)$.
Making implicit use of the musical isomorphisms here and throughout the text,
write $\nu^\eps(x) = \frac{\nabla u(x)}{\abs{\nabla u(x)}}$, so that
\begin{equation}
A^\Sigma(X,Y) = \langle \nabla_X Y,\nu^\eps \rangle \nu^\eps
= - \langle Y, \nabla_X \nu^\eps \rangle \nu^\eps.
\end{equation}
\begin{lem}
	\label{lem:second_ff_ae_bound}
	Let $x \in U$ be a regular point of $u$. 
	Then
	\begin{equation}
	\label{eq:ae_bounded_hessian}
	\abs{\Ae}(x)^2
	\leq \frac{1}{\abs{\nabla u}^2(x)}
	(\abs{\nabla^2 u}^2(x) - \abs{\nabla \abs{\nabla u}}^2(x)),
	\end{equation}
	where $\nabla^2 u(x)$ is the Hessian of $u$ at $x$.
\end{lem}
\begin{proof}
	The second fundamental form $A^\Sigma$ is expressed in terms of the
	covariant derivative $\nabla \nu^\eps$ by
	\begin{equation}
	A^\Sigma = - \srestr{\nabla \nu^\eps}{T\Sigma \otimes T\Sigma} \otimes \nu^\eps.
	\end{equation}
	We can express $\nabla \nu^\eps$ as
	\begin{equation}
	\nabla \nu^\eps =
	\frac{\nabla^2 u}{\abs{\nabla u}} -
	\nu^\eps \otimes \frac{\nabla \abs{\nabla u}}{\abs{\nabla u}},
	\end{equation}
	whence after restriction to $T\Sigma \otimes T\Sigma$ we get
	\begin{equation}
	A^\Sigma =
	- \frac{1}{\abs{\nabla u}}
	\srestr{\nabla^2 u}{T\Sigma \otimes T\Sigma}
	\otimes \nu^\eps.
	\end{equation}
	On the other hand $\nabla \abs{\nabla u} = \langle \nabla^2 u, \nu^\eps \rangle$
	where $\nabla u \neq 0$, so after decomposing the Hessian $\nabla^2 u$
	in terms of its action on $T\Sigma$ and $N\Sigma$, we obtain
	\begin{equation}
	\abs{\nabla^2 u}^2 - \abs{\nabla \abs{\nabla u}}^2
	= \abs{\nabla u}^2 \abs{A^\Sigma}^2 +
	\abs{\srestr{\nabla^2 u}{T\Sigma \otimes N\Sigma}}^2
	\geq \abs{\nabla u}^2 \abs{A^\eps}^2. \qedhere
	\end{equation}
\end{proof}

When considering the
second variation, it somewhat simplifies notation to rescale the energy as
$\calE_\eps = \eps^{-1} E_\eps$.
Its second variation is
$ \delta^2 \calE_\eps(u)(\phi,\phi) =
\int_U \abs{\nabla \phi}^2 + \frac{W''(u)}{\eps^2} \phi^2, $
defined for all $\phi \in C_c^1(U)$, which by a density argument can be extended to $H_0^1(U)$.
The following identity will be useful throughout; a proof can be found
in either of the indicated sources.
\begin{lem}[\cite{Farina13,Tonegawa05}]
\label{lem:expr_second_variation}
Let $u \in C^3(U) \cap L^\infty(U)$ be a critical point of $E_\eps$.
For all $\phi \in C_c^1(U)$,
\begin{multline}
	\label{eq:second_variation_rescaled_expression}
	\delta^2 \calE_\eps(u)(\abs{\nabla u} \phi,\abs{\nabla u} \phi)
	= \\	\int_{U} \abs{\nabla u}^2 \abs{\nabla \phi}^2
	- \left(\abs{\nabla^2 u}^2 -
	\abs{\nabla \abs{\nabla u}}^2 + \Ric_M(\nabla u,\nabla u) \right)
	\phi^2 \dH^{n+1}.
\end{multline}
\end{lem}
Combining \eqref{eq:second_variation_rescaled_expression} 
with the $\nve$-a.e. bound \eqref{eq:ae_bounded_hessian} yields
for all $\phi \in C_c^1(U)$
\begin{equation}
\label{eq:second_inequality_stability}
\frac{\eps}{2 \sigma}	
\delta^2 \calE_{\eps}(u)(\abs{\nabla u}\phi,\abs{\nabla u}\phi)
\leq
\int \abs{\nabla \phi}^2
- (\abs{A^\eps}^2 + \Ric_M(\nu^\eps,\nu^\eps)) \phi^2 \dnve.
\end{equation}
When $u$ is stable, that is when $\delta^2 \calE_\eps(u)$ is non-negative,
then this identity yields $L^2(V^\eps)$--bounds for $A^\eps$.
\begin{cor}
	\label{cor:L2_bound_for_Ae_and_Be_if_u_stable}
	There is a constant $C = C(M) > 0$ such that if $u \in C^3(U) \cap L^\infty(U)$
	is a critical point of $E_\eps$ and is stable in an open ball $B(x,r) \subset U$
	of radius $r \leq 1$ then
	\begin{equation}
	\label{eq:bound_in_cor_stability_in_balls}
	\int_{B(x,\frac{r}{2})} \abs{\Ae}^2  \dnve
	\leq \frac{C}{r^2}  \norm{\ve}(B(x,r)).
	\end{equation}
\end{cor}
\begin{proof}
	The Ricci curvature
	term in \eqref{eq:second_inequality_stability} can be bounded by some
	constant $C(M) \geq 1$ as the ambient manifold $M$ is closed, so
	$ \int_{B(x,r)} \abs{\Ae}^2 \phi^2  \dnve
	\leq C(M) \int \phi^2 + \abs{\nabla \phi}^2 \dnve $
	for all $\phi \in C_c^1(B(x,r))$. Plug in a cut-off function
	$\eta \in C_c^1(B(x,r))$ with $\eta = 1$ in $B(x,\frac{r}{2})$ and
	$\abs{\nabla \eta} \leq 3r^{-1}$ to obtain the desired inequality.
\end{proof}
We now turn to a sequence $(u_i)$ of critical points satisfying
Hypotheses~\ref{hypA}--\ref{hyp:C_index_bound}.
If the $u_i$ are stable in a ball as in
Corollary~\ref{cor:L2_bound_for_Ae_and_Be_if_u_stable},
then the uniform weight bounds \eqref{eq:density_weight_measure}
imply uniform $L^2(V^i)$--bounds of the second fundamental forms,
which we denote $A^i$ from now on.
Under these conditions the $A^i$ converge
weakly to the second fundamental form $A$ (in the classical, smooth
sense) of $\reg V$.
\begin{prop}
	\label{prop:weak_convergence_second_ff}
	Let $W \subset \subset U \setminus \sing V$ be open with
	$W \cap \reg V \neq \emptyset$.
	If $\sup_i \int_W \abs{\Aei}^2 \dnvi < +\infty$,
	then passing to a subsequence $\Aei \dvei \rightharpoonup A \dv$ weakly as Radon
	measures on $G_n(W)$, and
	\begin{equation}
	\label{eq:bound_fatou_second_ff}
	\int_W \abs{A}^2 \dnv \leq
	\liminf_{i \to \infty} \int_W \abs{\Aei}^2 \dnvi,
	\end{equation}
	where $A$ is the second fundamental form of $\reg V \subset M$.
\end{prop}
The weak subsequential convergence follows immediately from compactness
of Radon measures; the main difficulty is to show that the weak limit is $A \dv$.
The proof is a straight-forward adaptation of the argument
used for the stable case in~\cite{Tonegawa05}; we present a
complete argument in Appendix~\ref{app:second_fundamental_form_coordinate_expression}
for the reader's convenience.

\subsection{Spectrum of \texorpdfstring{$L_V$}{LV} and weighted min-max}
\label{sec:prelim_spectrum_sec_var}

Throughout we restrict ourselves to a compactly contained open
subset $W \subset \subset U \setminus \sing V$
to avoid blow-up of the coefficients of $L_V$ near the singular set,
and assume $W \cap \reg V \neq \emptyset$ to avoid vacuous statements.
As $W \cap \reg V$ is compactly contained in $\reg V$, it can
intersect only finitely many connected components $C_1,\dots,C_N$
of $\reg V$.
By the constancy theorem~\cite[Thm.~41.1]{Simon84}
the multiplicity function $\Theta$ of a stationary
integral varifold $V$ is constant on every connected component of $\reg V$;
we write $\Theta_1,\dots,\Theta_N$ for the respective multiplicities
of $C_1,\dots,C_N$.

By classical theory for elliptic PDE~\cite[Ch.~8]{GilbargTrudinger98},
the spectrum of $L_V$ has the following min-max characterisation:
\begin{equation}
\label{eq:unw_min_max}
\lambda_p = \inf_{\dim S = p} \max_{\phi \in S \setminus \{ 0 \}}
\frac{B_{\reg V}(\phi,\phi)}{\norm{\phi}^2_{L^2}}
\quad \text{for all $p \in \NN$},
\end{equation}
where the infimum is taken over linear subspaces $S$
of $H_0^1$ (recall this is our
abbreviated notation for $H_0^1(W \cap \reg V)$).
From this we easily obtain a min-max characterisation that is `weighted' by
the multiplicities  $\Theta_1,\dots,\Theta_N$ in the sense that
\begin{equation}
\label{eq:w_min_max}
\lambda_p = \inf_{\dim S = p} \max_{\phi \in S \setminus \{0 \}}
\frac{B_V(\phi,\phi)}{\norm{\phi}^2_{L^2(V)}}
\quad \text{for all $p \in \NN$}.
\end{equation}
To see this, observe the following: as functions $\phi \in H_0^1$
vanish near the boundary of every connected component $C \subset \reg V$,
the function $\phi_C$ on $W \cap \reg V$ defined by
\begin{equation}
\phi_C = \begin{cases}
\phi &\text{on }  C \\
0 & \text{on } W \cap \reg V \setminus C
\end{cases}
\end{equation} also belongs to $H_0^1$.
Moreover 
\begin{equation}
B_V(\phi_C,\phi_C)
= \Theta_C B_{\reg V}(\phi_C,\phi_C)
\text{ and }
\norm{\phi_C}^2_{L^2(V)}
= \Theta_C \norm{\phi_C}^2_{L^2},
\end{equation}
where $\Theta_C$ denotes the multiplicity of $C$.
We then define a linear isomorphism of $H_0^1$ via normalisation
by the respective multiplicities of the components. This sends
$\phi \mapsto \bar{\phi} := \sum_{j=1}^N \Theta_j^{-1/2} \phi_{C_j}$; then
\begin{equation}
\frac{B_V(\bar{\phi},\bar{\phi})}{\norm{\bar{\phi}}_{L^2}^2}
= \frac{B_{\reg V}(\phi,\phi)}{\norm{\phi}_{L^2(V)}^2}.
\end{equation}
Therefore the `unweighted' and `weighted' min-max
characterisations \eqref{eq:unw_min_max} and \eqref{eq:w_min_max} are in fact equivalent.
In the remainder we mainly use \eqref{eq:w_min_max}, and
abbreviate this as $\lambda_p = \inf_{\dim S = p} \max_{S \setminus \{0\}} J_V$,
where $J_V$ denotes the `weighted' Rayleigh quotient
\begin{equation}
J_V(\phi) = \frac{B_V(\phi,\phi)}{\norm{\phi}_{L^2(V)}^2}
\text{\quad for all $\phi \in H_0^1 \setminus \{0\}$}.
\end{equation}

The min-max characterisation implies the following lemma, which
highlights the dependence of the spectrum $\lambda_p(W)$ on
the subset $W$.
\begin{lem}
	\label{lem:shrinking_balls_spectrum}
	\leavevmode
	\begin{enumerate}[label = (\alph{*}), font = \upshape]
		\item If $W_1 \subset W_2 \subset \subset U \setminus \sing V$, then
		$\lambda_p(W_1) \geq \lambda_p(W_2)$: the spectrum
		is monotone decreasing.
		\item \label{item:index_superadd} If $W_1,W_2 \subset \subset U \setminus \sing V$
		have $W_1 \cap W_2 = \emptyset$, then
		$\ind_{W_1} B_V + \ind_{W_2} B_V = \ind_{W_1 \cup W_2} B_V$.
		\item \label{item:shrinking_balls}
		If $W \subset \subset U \setminus \sing V$ and $y \in W \cap \reg V$, then
		$\lambda_p(W) = \lim_{R \to 0} \lambda_p(W \setminus \overline{B}(y,R))$.
	\end{enumerate}
	
\end{lem}
\begin{rem}
	The same properties hold for the spectrum and index of $L_i$, and the proof is easily
	modified to cover this case.
\end{rem}
\begin{proof}
(a) This is immediate from the min-max characterisations,
or simply by definition of the spectrum. Similarly for \ref{item:index_superadd}.

\ref{item:shrinking_balls} By monotonicity of the spectrum we have
\begin{equation}
\lambda_p(W \setminus \overline{B}(y,R))
\geq \lambda_p(W \setminus \overline{B}(y,R'))
\geq \lambda_p(W)
\end{equation}
for all $R > R' > 0$. The limit as $R \to 0$ therefore exists and is bounded below
by $\lambda_p(W)$; it remains only to show that $\lim_{R \to 0}
\lambda_p(W \setminus \overline{B}(y,R)) \leq \lambda_p(W)$.

By monotonicity of the spectrum it is equivalent to show that
for a fixed radius $R > 0$, $\lim_{m \to \infty}
\lambda_p(W \setminus \overline{B}(y,2^{-m}R))
\leq \lambda_p(W)$.
Let $(\rho_m)_{m \in \NN}$ be a sequence
in $C_c^1(B(y,R) \cap \reg V)$ with the following properties (such a sequence
exists provided $n \geq 2$, see Remark~\ref{rem:cutoff_functions} below):
\begin{enumerate}[label = (\arabic*)]
	\item \label{item:cutoff_props_a} $\srestr{\rho_m}{B(y,2^{-m}R) \cap \reg V}
	\equiv 0$ and 	$\rho_m \to 1$ $\calH^n$-a.e.\ in
	$W \setminus \{ y \} \cap \reg V $,
	
	\item \label{item:cutoff_props_b}
	$\norm{\nabla_V \rho_m}_{L^2(W \cap \reg V)} \to 0$.
\end{enumerate}
Let a small $\delta > 0$ be given and choose a family $(\phi_1,\dots,\phi_p)$
in $C_c^1(W \cap \reg V)$ whose $\linspan(\phi_1,\dots,\phi_p) =: S$ has
$\max_{S \setminus 0} J_V \leq \lambda_p(W) + \delta$.
Write $\rho_m S$ for 
$\linspan(\rho_m \phi_1,\dots,\rho_m \phi_p)\subset
C_c^1(W \setminus \overline{B}(y,2^{-m}R) \cap \reg V)$---for
$m$ large enough the functions $\rho_m \phi_i$ are indeed linearly independent.
By the weighted min-max formula \eqref{eq:w_min_max},
\begin{equation}
\max_{\rho_m S \setminus 0} J_V
\geq \lambda_p(W \setminus \overline{B}(y,2^{-m}R)).
\end{equation}

Let $t_m \in \Sph^{p-1} \subset \RR^p$ denote the coefficients of the linear combination
$t_m \cdot \rho_m \phi := \rho_m \sum_{j=1}^p t_{mj} \phi_j \in \rho_m S$
that realises $\max_{\rho_m S \setminus 0} J_V$.
Passing to a convergent subsequence $t_m \to t \in \Sph^{p-1} \subset \RR^p$
we get $J_V(t_m \cdot \rho_m \phi) \to J_V(t \cdot \phi)$, and hence
\begin{equation}
\lim_{m \to \infty} J_V(t_m \cdot \rho_m \phi)
\leq \max_{S \setminus 0} J_V.
\end{equation}
On the one hand $\max_{S \setminus 0} J_V \leq \lambda_p(W) + \delta$ by
our choice of $S$,
on the other hand $\lim_{m \to \infty} \lambda_p(W \setminus \overline{B}(y,2^{-m}R))
\leq \lim_{m \to \infty} J_V(t_m \cdot \rho_m \phi)$ by our choice of $t_m$.
The conclusion follows after combining these two observations and letting $\delta \to 0$.
\end{proof}

\begin{rem}
\label{rem:cutoff_functions}
A sequence of functions $(\rho_m)$ with properties
\ref{item:cutoff_props_a} and \ref{item:cutoff_props_b} exists provided $n \geq 2$,
as we assume throughout. When $n \geq 3$ one can use the standard cutoff functions;
for $n = 2$ a more precise construction is necessary, described
for instance in~\cite[Sec.~4.7]{Evans15}.
\end{rem}

\subsection{Spectrum of \texorpdfstring{$L_i$}{Li} and conditional proof of
	Theorem~\texorpdfstring{\ref{thm:main_thm_propagation_of_index_bounds}}{A}}
The main result in this section is Lemma~\ref{lem:theorem_holds_if_second_ff_bounded};
essentially it says that 
\begin{equation}
\lambda_p(W) \geq \limsup_{i \to \infty} \lambda_p^i(W)
\end{equation}
holds under the condition that $\sup_i \int_W \abs{A^i}^2 \dnvi < +\infty$.
What precedes it in this section are technical results required for its proof.

Again, by classical elliptic PDE theory the eigenvalues $\lambda_p^i(W)$
of $L_i = \Delta - \eps_i^{-2} W''(u_i)$ on $H_0^1(W)$
have the following min-max characterisation
in terms of the rescaled Allen--Cahn functional
$\calE_{\eps_i} = \eps_i^{-1} E_{\eps_i}$:
\begin{equation}
\label{eq:min_max_L_i}
\lambda_p^i(W) = \inf_{\dim S = p} \max_{\phi \in S \setminus \{0\}}
\frac{\delta^2 \calE_{\eps_i}(u_i)(\phi,\phi)}{\norm{\phi}_{L^2}^2}
\quad \text{for all $p \in \NN$},
\end{equation}
where the infimum is over $p$-dimensional linear subspaces
$S \subset H_0^1(W)$.
Define the \emph{Rayleigh quotient} $J_i$ by
\begin{equation}
J_i(\phi) = \frac{\delta^2 \calE_{\eps_i}(u_i)(\phi,\phi)}{\norm{\phi}_{L^2}^2}
\quad \text{for all $\phi \in H_0^1(W) \setminus \{0\}$},
\end{equation}
so that we can write the min-max characterisation more succinctly as
$\lambda_p^i = \inf_{\dim S = p} \max_{S \setminus \{0\}} J_i.$

To compare the spectrum of $L_V$ in $H_0^1$ with those of the
operators $L_i$ in $H_0^1(W)$,
extend functions in $C_c^1(W \cap \reg V)$ to $C_c^1(W)$
in the standard way, which we now describe to fix notations.
Pick a small enough $0 < \tau < \inj(M)$ so that $B_\tau V :=
\exp N_\tau V$ is a tubular neighbourhood of $W \cap \reg V$, where
$N_\tau V := \{ s_p \in NV \mid p \in W \cap \reg V, \abs{s_p} < \tau \}$.
We abuse notation slightly to denote points in $B_\tau V$
by $s_p$, and identify the fibre $N_pV$ with
$(\exp_p)_* N_p V \subset T_{s_p} (B_\tau V)$.
The distance function $d_V: x \in B_\tau V \mapsto \dist(x,\reg V)$
is Lipschitz and smooth on $B_\tau V \setminus \reg V$.
By the Gauss lemma $\grad d_V(s_p) = -s_p/\abs{s_p}$ for all
$s_p \in B_\tau V \setminus \reg V$.
A function $\phi \in C^1(B_\tau V)$ is constant
along geodesics normal to $\reg V$ if $\phi(s_p) =
\phi(0_p)$ for all $s_p \in B_\tau V$,
or equivalently if $\langle \nabla \phi,\nabla d_V  \rangle \equiv 0$
in $B_\tau V \setminus \reg V$.

\begin{lem}
	\label{lem:extension_scalar_functions}
	Any $\phi \in C_c^1(W \cap \reg V)$ can be extended to $C_c^1(W)$
	with $\langle \nabla \phi, \nabla d_V \rangle \equiv 0$
	in $B_{\frac{\tau}{2}} V \setminus \reg V$ for some $\tau = \tau(\phi) > 0$.
\end{lem}
\begin{proof}
	Extend $\phi \in C_c^1(W \cap \reg V)$ to 
	$B_\tau V$ by setting $\tilde{\phi}(s_p) = \phi(p)$,  so that
	$\langle \nabla d_V , \nabla \tilde{\phi} \rangle \equiv 0$ in
	$B_\tau V \setminus \reg V$.
	Let $\eta \in C^1[0,\infty)$ be a cutoff
	function with $0 \leq \eta \leq 1$, $\eta \equiv 1 $ on $[0,1/2)$
	and $\spt \eta \subset [0,1)$.
	Then
	\begin{equation}
	(\eta \circ d_V/\tau) \tilde{\phi} \in C_c^1(B_\tau V)
	\text{ and }  (\eta \circ d_V/\tau) \tilde{\phi} = \tilde{\phi} \text{ on } B_{\tau/2} V.
	\end{equation}
	Moreover even though $B_\tau V \not\subset W$ in general, as $\spt \phi$
	is compactly contained in $W \cap \reg V$ we still have
	$ (\eta \circ d_V/\tau) \tilde{\phi} \in C_c^1(W)$ provided
	$0 < \tau < \dist (\spt \phi, \partial W)$.
\end{proof}
The following lemma gives an asymptotic lower bound for the Rayleigh
quotient $J_V$ in terms of the $J_i$.
\begin{lem}
	\label{lem:min_max_convergence_when_C_c_convergence}
	Let $B_{\tau}V$
	be a tubular neighbourhood of $W \cap \reg V$ with width $\tau > 0$,
	and let $(\phi_i)_{i \in \NN}$ be a sequence of functions in $C_c^1(W)$ with
	\begin{enumerate}[label = (\alph*), font = \upshape]
		\item $\langle \nabla \phi_i,\nabla d \rangle \equiv 0$ in $W \cap B_{\tau/2}V$
		for all $i$,
		\item $\phi_i \to \phi$ in $C_c^1(W)$ as $i \to \infty$,
		where $\phi \neq 0$ in $W \cap \reg V$.
		\label{item:Cc_cvg_second_assumption}
	\end{enumerate}
	If $\sup_i \int_W \abs{A^i}^2 \dnvi < +\infty$,
	then $J_V(\phi) \geq \limsup_{i \to \infty}
	J_i(\abs{\nabla u_i}\phi_i)$.
\end{lem}
\begin{proof}
Before we start the proof proper, note that for all $\phi \in H_0^1(U)$,
dividing both sides of \eqref{eq:second_inequality_stability}
by $ \frac{\eps_i}{2 \sigma} \int \phi^2 \abs{\nabla u_i}^2 \dH^{n+1}
= \int \phi^2 \dnvi$ yields
\begin{equation}
\label{eq:sec_var_sec_form}
\norm{\phi}_{L^2(\vi)}^{-2}
\int \abs{\nabla \phi}^2 - (\abs{A^i}^2 + \Ric_M(\nui,\nui)) \phi^2 \dnvi
\geq J_i(\abs{\nabla u_i}\phi),
\end{equation}
provided of course that $\norm{\phi}_{L^2(V^i)}^2 \neq 0$.

We treat the terms on the left-hand side separately in the calculations
\eqref{item:proof_thmA_i_1}--\eqref{item:proof_thmA_i_4} below.
Once these are completed, we combine \eqref{item:proof_thmA_i_1} with our
assumption that $\norm{\phi}_{L^2(V)}^2 \neq 0$ to obtain that
$\norm{\phi_i}_{L^2(V^i)} \neq 0$ for large enough $i$.
The conclusion follows by combining \eqref{eq:sec_var_sec_form}
with the remaining calculations: 
%
\begin{enumerate}
	\item $\int \phi^2 \dnv = \lim_{i \to \infty} \int \phi_i^2
	\dnvi$\label{item:proof_thmA_i_1}
	\item $\int \abs{\nabla_V \phi}^2 \dnv = \lim_{i \to \infty} \int
	\abs{\nabla \phi_i}^2 \dnvi$, \label{item:proof_thmA_i_2}
	\item $\int \abs{A}^2 \phi^2 \dnv
	\leq \liminf_{i \to \infty} \int \abs{A^i}^2 \phi_i^2 \dnvi$,
	\label{item:proof_thmA_i_3}
	\item $\int \Ric_M(\nu,\nu) \phi^2 \dnv
	= \lim_{i \to \infty } \int \Ric_M(\nui,\nui)\phi_i^2 \dnvi$.
	\label{item:proof_thmA_i_4}
\end{enumerate}

\eqref{item:proof_thmA_i_1}
By assumption $\phi_i^2 \to \phi^2$
in $C_c(W)$, whence by Corollary~\ref{cor:meas_fct_cvg} we get
$\int \phi^2 \dnv = \lim_{i \to \infty} \int \phi_i^2 \dnvi$.
The same argument proves \eqref{item:proof_thmA_i_2}, after noticing
that $\langle \nabla \phi_i, \nabla d_V \rangle \equiv 0$ implies
$\abs{\nabla \phi}^2 = \abs{\nabla_V \phi}^2$ on $W \cap \reg V$.

\eqref{item:proof_thmA_i_3}
The sequence $(A^i \phi_i \dnvi)$ converges weakly to $A \phi \dnv$,
as we can show by
testing against an arbitrary $\varphi \in C_c(U)$:
\begin{multline}
\int A^i \phi_i \varphi \dnvi 
- \int A \phi \varphi \dnv
= \\ \int A^i (\phi_i - \phi) \varphi \dnvi
+ \int A^i \phi \varphi \dnvi - \int A \phi \varphi \dnv.
\end{multline}
The first integral is bounded by
\begin{equation}
\Big \lvert \int A^i (\phi_i - \phi) \varphi \dnvi \Big \rvert
\leq 
\norm{\phi_i - \phi}_{L^\infty} \norm{\varphi}_{L^2(\vi)}
\norm{A^i}_{L^2(V^i)} \to 0
\end{equation}
because $\phi_i \to \phi$
in $C_c(W)$ as $i \to \infty$.
The remaining terms tend to $0$ by the weak convergence of
$A^i \dnvi \rightharpoonup A \dnv$ tested against
$\phi \varphi \in C_c(W)$.
Then inequality \eqref{eq:measure_function_convergence_L2_bound}
gives
$
\int \abs{A}^2 \phi^2 \dnv
\leq \liminf_{i \to \infty} \int \abs{A^i}^2 \phi_i^2 \dnvi.
$

\eqref{item:proof_thmA_i_4}
For each $S \in G_n(T_p M)$ pick a unit vector $\nu_S$ in $T_p M$
orthogonal to $S$, and define a smooth function $R_M$ on $G_n(U)$ by
$R_M \! : S \mapsto \Ric_M(\nu_S,\nu_S)$.
Then $\phi_i^2 R_M \to \phi^2 R_M$ in $C_c(G_n(U))$ as $i \to \infty$,
and by Corollary~\ref{cor:meas_fct_cvg},
\begin{multline}
\int  \phi_i^2 \Ric_M(\nui,\nui) \dnvi
= \\  \int \phi_i^2 R_M \dvi 
\to 
\int \phi^2 R_M \dv = \int \phi^2 \Ric_M(\nu,\nu) \dnv.
\qedhere
\end{multline}
\end{proof}

We conclude the section with a proof of
Theorem~\ref{thm:main_thm_propagation_of_index_bounds}\ref{item:main_thm_spec}
in the case where there is a uniform $L^2(\vei)$--bound on the second
fundamental forms $(\Aei)$.
\begin{lem}
	\label{lem:theorem_holds_if_second_ff_bounded}
	Let $W \subset \subset U \setminus \sing V$ be open with
	$W \cap \reg V \neq \emptyset$.
	If $\sup_i \int_W \abs{\Aei}^2 \dnvi < +\infty$,
	then $\lambda_p(W) \geq \limsup_{i \to \infty} \lambda_p^i(W)$
	for all $p$.
\end{lem}
\begin{proof}
We may assume that every connected component of $W$ intersects $\spt \nv$
(or $\reg V$, equivalently as $W \cap \sing V = \emptyset$)
without restricting generality: if $C$ is a
connected component of $W$ with $C \cap \reg V = \emptyset$,
then $\lambda_p(W \setminus C) = \lambda_p(W)$, and by monotonicity
$\lambda_p^i(W \setminus C) \geq \lambda_p^i(W)$ for all $i$.

Given $\delta > 0$ there is a $p$-dimensional linear subspace
$S = \linspan(\phi_1,\dots,\phi_p)$ of $C_c^1(W \cap \reg V)$ with
\begin{equation}
\label{eq:almost_maximiser_but_in_C_c}
\lambda_p(W) + \delta
\geq
\max_{S \setminus \{0\}} J_V.
\end{equation}
Extend the functions $\phi_i$
to $C_c^1(W)$ as in Lemma~\ref{lem:extension_scalar_functions};
for large enough $i$ the family
$(\abs{\nabla u_i}\phi_1,\dots,\abs{\nabla u_i} \phi_p)$
is still linearly independent.
Indeed, otherwise we could extract a subsequence
such that $(\abs{\nabla u_{i'}} \phi_1,\dots,
\abs{\nabla u_{i'}} \phi_p)$ has a linear dependence, with coefficients
$a_{i'} \neq 0 \in \RR^p$ say.
Then notice that 
\begin{equation}
\abs{\nabla u_{i'}} a_{i'} \cdot  \phi = 0
\Leftrightarrow \norm{a_{i'} \cdot \phi}_{L^2(V^{i'})} = 0,
\end{equation}
where we abbreviated
$a_{i'} \cdot \phi := \sum_{j=1}^p a_{i'j} \abs{\nabla u_{i'}} \phi_j$.
We may normalise the coefficients $a_{i'}$ so as to guarantee
$\abs{a_{i'}} = 1$ and then, possibly after extracting a second subsequence, assume that $a_{i'} \to a \in \Sph^{p-1}$ as $i' \to \infty$.
The resulting strong convergence
$a_{i'} \cdot \phi \to a \cdot \phi$ in $C_c(W)$
combined with $\nvi \rightharpoonup \nv$ yield
$\norm{a_{i'} \cdot \phi}_{L^2(V^{i'})} \to \norm{a \cdot \phi}_{L^2(V)}$;
this contradicts $\norm{a \cdot \phi}_{L^2(V)} > 0$.

From now on take $i$ large enough so that
$(\abs{\nabla u_i} \phi_1,\cdots,\abs{\nabla u_i} \phi_p)$
is linearly independent.
For such large $i$, we may let $t_i \in \Sph^{p-1} \subset \RR^p$
be the (normalised) coefficients of a linear combination
$t_i \cdot \abs{\nabla u_i} \phi
= \sum_{j=1}^p t_{ij} \abs{\nabla u_i} \phi_j$
that maximises the Rayleigh quotient $J_i$:
\begin{equation}
J_i(t_i \cdot \abs{\nabla u_i} \phi)
= \max_{\abs{\nabla u_i} S \setminus \{0\}} J_i \geq \lambda_p^i(W).
\end{equation}
Extract a convergent subsequence $t_{i'} \to t \in \Sph^{p-1}$,
so that $t_{i'} \cdot \phi \to t \cdot \phi$ in $C_c^1(W)$ as $i' \to \infty$.
Lemma~\ref{lem:min_max_convergence_when_C_c_convergence} gives
$J_V(t \cdot \phi) \geq \limsup_{i \to \infty}
\max_{\abs{\nabla u_i} S\setminus \{ 0\}} J_i,$
which in turn is greater than $\limsup_{i \to \infty} \lambda_p^{i}(W)$.
Using \eqref{eq:almost_maximiser_but_in_C_c},
\begin{equation}
\lambda_p(W) + \delta
\geq J_V(t \cdot \phi)
\geq \limsup_{i \to \infty} \lambda_p^i(W),
\end{equation}
and we conclude by letting $\delta \to 0$.
\end{proof}

Lemma~\ref{lem:theorem_holds_if_second_ff_bounded}
has the following immediate corollary.
\begin{cor}
	\label{cor:thmA_ii_when_second_ff_is_bounded}
	Under the hypotheses of Lemma~\ref{lem:theorem_holds_if_second_ff_bounded},
	\begin{equation}
	\ind_W B_V \leq \liminf_{i \to \infty} (\ind_W u_i).
	\end{equation}
\end{cor}

\section{Proof of the main theorem (Theorem A)}
\label{sec:proof_main_thm}
We briefly recall the context of the proof: $M^{n+1}$ is a closed Riemannian
manifold and $U \subset M$ is an arbitrary open subset. The sequence
of functions $(u_i)$ in $C^3(U)$ satisfies Hypotheses \ref{hypA}, \ref{hypB}
and \ref{hyp:C_index_bound}---the last hypothesis says that
$\ind u_i \leq k$ for all $i$. 
To every $u_i$ we associate the $n$--varifold $\vei$ from
\eqref{eq:definition_varifolds}. By Theorem~\ref{lem:properties_V},
we may pass to a subsequence of $(\vei)$ converging
weakly to a stationary integral varifold $V$. 
\subsection{Spectrum and index of \texorpdfstring{$V$}{V}: proof of (i) and (ii)}
The main idea, inspired by an argument of Bellettini--Wickramasekera
\cite{BelWic16}, is to fix a compactly contained open subset
 $W \subset \subset U \setminus \sing V$ and study
the stability of $u_i$ in open balls covering $W \cap \reg V \neq \emptyset$.
We then shrink the radii of the covering balls to $0$, and prove
the spectral lower bound of
Theorem~\ref{thm:main_thm_propagation_of_index_bounds}\ref{item:main_thm_spec} by 
induction on $k$. The upper bound on $\ind B_V$ of
Theorem~\ref{thm:main_thm_propagation_of_index_bounds}\ref{item:main_thm_ind_reg}
is then an immediate consequence.

In the base of induction the $u_i$ are stable in $U$.
Let $\eta \in C_c^1(U)$ be a cutoff function constant equal to $1$ on $W$.
The stability inequality \eqref{eq:second_inequality_stability} gives that
\begin{equation}
\int_W \abs{\Aei}^2 \dnvi
\leq C(M) \dist(W,\partial U)^{-2}
\norm{\vei}(U) \quad \text{for all $i$}.
\end{equation}
Combining this with \eqref{eq:density_weight_measure}
we get $\sup_i \int_W \abs{\Aei}^2 \dnvi < +\infty$,
and thus $\lambda_p(W) \geq \limsup_{i \to \infty} \lambda_p^i (W)$
by Lemma~\ref{lem:theorem_holds_if_second_ff_bounded}.

For the induction step, let $k \geq 1$ and assume that
Theorem~\ref{thm:main_thm_propagation_of_index_bounds}\ref{item:main_thm_spec}
holds with $k-1$ in place of $k$. Consider an arbitrary $W \subset \subset
U \setminus \sing V$ that intersects $\reg V$. 
Fix a radius $0 < r < \dist(W,\sing V)$, and pick points
$x_1,\dots,x_N \in W \cap \reg V$ such that
$W \cap \reg V \subset \cup_{j = 1}^N B(x_j,\frac{r}{2})$.
We define the following \emph{Stability Condition} for the cover
$\{B(x_j,\frac{r}{2})\}_{1 \leq j \leq N}$:
\begin{enumerate}[label={(SC)}]
	\item \emph{For large $i$, each $u_i$ is stable in every ball
		$B(x_1,r),\dots,B(x_N,r)$. \label{item:stability_condition}}
\end{enumerate}

\begin{claim}
\label{claim:SH_regularity}
If for the cover $\{B(x_j,\frac{r}{2})\}_{1 \leq j \leq N}$:
\begin{enumerate}[label = (\alph*)]
\item \label{item:reg_SH_holds}
\ref{item:stability_condition} holds,
then $\lambda_p(W) \geq \limsup_{i \to \infty} \lambda_p^i(W)$;

\item \label{item:reg_SH_fails}
\ref{item:stability_condition} fails, then $\lambda_p(W \setminus \overline{B_r})
\geq \limsup_{i \to \infty} \lambda_p^i(W \setminus \overline{B_r})$
for some ball $B_r\in \{ B(x_j,r)\}$.
\end{enumerate}
\end{claim}
\begin{proof}
\ref{item:reg_SH_holds}
Let $W_r = W \cap \cup_{j=1}^N B(x_j,\frac{r}{2})$, so that 
$W_r \cap \reg V = W \cap \reg V$ and hence $\lambda_p(W_r) = \lambda_p(W)$.
Moreover $W_r \subset W$, so $\lambda_p^i(W_r) \geq \lambda_p^i(W)$
for all $i$ by monotonicity of the spectrum.
Therefore it is enough to show that $\lambda_p(W_r)
\geq \limsup_{i \to \infty} \lambda_p^i(W_r)$.

Because \ref{item:stability_condition} holds, summing
\eqref{eq:bound_in_cor_stability_in_balls} over all balls we get
\begin{equation}\label{bound}
\int_{W_r} \abs{\Aei}^2 \dnvi \leq
\frac{NC}{r^2} \norm{\vei}(W_r) \leq \frac{NCE_{0}}{2 r^2 \sigma}
\quad \text{for all $i$},
\end{equation}
so $\lambda_p(W_r) \geq \limsup_{i \to \infty} \lambda_p^i(W_r)$
by Lemma~\ref{lem:theorem_holds_if_second_ff_bounded}.

\ref{item:reg_SH_fails} If \ref{item:stability_condition} fails, then
some subsequence $(u_{i'})$ must be unstable in one of the balls
$B_r \in \{B(x_j,r)\}$,
in other words $\ind_{B_r} u_{i'} \geq 1$ for all $i'$.
On the other hand
\begin{equation}
\ind_{B_r} u_{i'} + \ind_{W \setminus \overline{B_r}} u_{i'}
\leq \ind_W u_{i'}
\end{equation}
because $B_r$ and $W \setminus \overline{B_r}$ are disjoint
open sets.
As $\ind_W u_{i'} \leq k$, we get 
$\ind_{W \setminus \overline{B_r}} u_{i'} \leq k-1$
for all $i'$, and we conclude after applying the
induction hypothesis to $(u_{i'})$ in $W \setminus \overline{B_r}$.\end{proof}
\begin{rem}
This argument shows that when \ref{item:stability_condition} fails there is
a ball $B_r\in \{B(x_j,r)\}$ with $\lambda_p(W \setminus \overline{B_r})
\geq \limsup_{i \to \infty} \lambda_{p+1}^i(W)$ for $p \geq k$, and thus also
$\ind_{W \setminus \overline{B_r}} B_V \leq k-1$, but
the induction step only requires the weaker conclusion
from Claim~\ref{claim:SH_regularity}.
\end{rem}

Now consider a decreasing sequence $r_m \to 0$ with
$0 < r_m < \dist(W,\sing V)$ and reason as above with $r = r_m$.
For each $m$, pick points $x_1^m,\dots,x_{N_{m}}^m \in W \cap \reg V$
such that
\begin{equation}
W \cap \reg V \subset \cup_{j=1}^{N_m} B \big(x_j^m,\frac{r_m}{2} \big).
\end{equation}
If \ref{item:stability_condition} holds for a cover
$\{B(x_j^m,\frac{r_m}{2})\}$, then
$\lambda_p(W) \geq \limsup_{i \to \infty} \lambda_p^i(W)$
by Claim~\ref{claim:SH_regularity}, and the induction step
is completed.

Otherwise \ref{item:stability_condition} fails for all constructed covers,
and by Claim~\ref{claim:SH_regularity} there is a sequence $(y_m)$ in
$W \cap \reg V$ with
\begin{equation}
\lambda_p (W \setminus \overline{B}(y_m,r_m))
\geq \limsup_{i \to \infty} \lambda_p^i(W \setminus \overline{B}(y_m,r_m)),
\end{equation}
and thus by monotonicity of the spectrum
\begin{equation}
\label{eq:spec_lower_bound_outside_ball}
\lambda_p (W \setminus \overline{B}(y_m,r_m)) \geq
\limsup_{i \to \infty} \lambda_p^i(W).
\end{equation}
Passing to a subsequence if necessary, we may assume that
$(y_m)$ converges to a point $y \in \overline{W} \cap \reg V$.
If we fix a radius $R > 0$, then $B(y_m,r_m) \subset B(y,R)$ for large enough $m$, 
so by monotonicity and \eqref{eq:spec_lower_bound_outside_ball},
\begin{equation}
\lambda_p(W \setminus \overline{B}(y,R))
\geq \limsup_{m \to \infty} \lambda_p(W \setminus \overline{B}(y_m,r_m)) 
\geq \limsup_{i \to \infty} \lambda_p^i(W).
\end{equation}
The conclusion follows after combining this with
$\lambda_p(W) = \lim_{R \to 0} \lambda_p(W \setminus \overline{B}(y,R))$
from Lemma~\ref{lem:shrinking_balls_spectrum}.

Together with the base of induction, we have proved
Theorem~\ref{thm:main_thm_propagation_of_index_bounds}\ref{item:main_thm_spec}
for all sequences $(u_i)$ with $\sup_i \ind u_i \leq k$ for some $k \in \NN$.
The index bound $\ind_W B_V \leq k$ follows immediately:
the spectral lower bound implies that $L_V$ must have fewer
negative eigenvalues than the $L_i$ as $i \to \infty$. Therefore
\begin{equation}
\ind_W B_V = \mathrm{card} \{p \in \NN \mid \lambda_p(W)  < 0 \} \leq k.
\end{equation}
As the subset $W$ was arbitrary we get $\ind B_V \leq k$;
this proves
Theorem~\ref{thm:main_thm_propagation_of_index_bounds}\ref{item:main_thm_ind_reg}.

\subsection{Regularity of \texorpdfstring{$V$}{V}: proof of \texorpdfstring{$\dimh \sing V \leq n-7$}{dimh sing V <= n-7}}
\label{subsection:proof_regularity_of_v}
The approach is the same as in the proof of
Theorem~\ref{thm:main_thm_propagation_of_index_bounds}\ref{item:main_thm_spec}--\ref{item:main_thm_ind_reg}
with one difference: we again proceed by induction on $k$, but
we now cover the entire support $\spt \nv$ (including the singular set),
instead of constructing covers a positive distance away from $\sing V$.

The base of induction, where the $u_i$ are stable in $U$,
was proved in~\cite{TonegawaWickramasekera10}.

For the induction step, suppose that $\dimh (U' \cap \sing V) \leq n-7$ 
holds with $k-1$ in place of $k$, and for arbitrary open subsets $U' \subset U$.
Fix $r > 0$, and choose points $x_1,\dots,x_N \in U \cap \spt \norm{V}$
such that $U \cap \spt \norm{V} \subset \cup_{j = 1}^N B(x_j,r)$.
The \emph{Stability Condition} for the cover
$\{B(x_j,r) \}_{1 \leq j \leq N}$ is defined in the same way as above,
except the radii need not be doubled:
\begin{enumerate}[label={(SC)}]
\item \emph{For large $i$, each $u_i$ is stable in every ball
	$B(x_1,r),\dots,B(x_N,r)$. \label{item:stability_condition_reg}}
\end{enumerate}
\begin{claim}
\label{claim:reg_SC_holds}
If for the cover $\{ B(x_j,r) \}_{1 \leq j \leq N}$:
\begin{enumerate}[label = (\alph*)]
	\item \label{item:reg_SC_holds} \ref{item:stability_condition_reg} holds,
	then $\dimh \sing V \leq n-7$;
	\item \label{item:reg_SC_fails} \ref{item:stability_condition_reg} fails,
	then $\dimh \sing V \setminus \overline{B_r} \leq n-7$
	for some ball $B_r\in \{B(x_j,r)\}$.
\end{enumerate}
\end{claim}
\begin{proof}
\ref{item:reg_SC_holds}
The results from \cite{TonegawaWickramasekera10} give
$\dimh B(x_j,r) \cap \sing V \leq n-7$ for every $j=1,\dots,N$.
As the balls $\{B(x_j,r)\}$ cover $U \cap \spt \norm{V}$, the same holds for
$\sing V$.

\ref{item:reg_SC_fails}
When \ref{item:stability_condition_reg} fails,
there must be a subsequence $(u_{i'})$ that is unstable in one of the balls
$B_r$ of the cover, so that in its complement
\begin{equation}
\ind_{U \setminus \overline{B_r}} u_{i'} \leq k-1
\quad \text{for all $i'$.}
\end{equation}
The conclusion follows from the induction hypothesis applied
to $(u_{i'})$ in $U \setminus \overline{B_r}$.
\end{proof}
Now consider a decreasing sequence $r_m \to 0$. For every $m$, choose
points $x_1^m,\dots,x_{N_m}^m \in U \cap \spt \norm{V}$ such that
$U \cap \spt \norm{V} \subset \cup_{j=1}^{N_m} B(x_j^m,r_m)$. 
Then either \ref{item:stability_condition_reg} holds for the cover
$\{ B(x_j^m,r_m)\}$ constructed for some $m$,
in which case we can conclude from
Claim~\ref{claim:reg_SC_holds},
or else there is sequence of points $(y_m)$ in $U \cap \spt \norm{V}$ for which
\begin{equation}
\dimh \sing V \setminus \overline{B}(y_m,r_m) \leq n-7.
\end{equation}
Possibly after extracting a subsequence, the sequence $(y_m)$ converges to a point
$y \in \overline{U} \cap \spt \norm{V}$.
As $U \setminus \{y\} \subset
\cup_{m \geq 0} U \setminus \overline{B}(y_m,r_m)$, we get
$\dimh \left(\sing V \setminus \{y\}\right) \leq n-7$.

If $n \geq 7$, then $\dimh \sing V \leq n-7$ holds whether or not $y \in \sing V$,
as points are zero-dimensional.
If however $2 \leq n \leq 6$ then we need $\sing V = \emptyset$, which
amounts to the following claim.

\begin{claim}
	\label{claim:y_not_in_sing_V}
	If $2 \leq n \leq 6$ then $y \notin \sing V$.
\end{claim}

\begin{proof}
Choose $B(y,R) \subset U$, and
consider balls $\{B(y,R_m)\}_{m \in \NN}$ with
shrinking radii $R_m := 2^{-m}R$.
If for some $m$ there is a subsequence $(u_{i'})$ with
\begin{equation}
\ind_{B(y,R_m)} u_{i'} \leq k-1
\quad \text{for all $i'$,}
\end{equation}
then we can conclude from the induction hypothesis.
Otherwise for all $m$
\begin{equation}
\ind_{B(y,R_m)} u_{i} = k
\quad \text{for $i$ large enough,}
\end{equation}
and the $u_i$ are eventually stable in the annulus $B(y,R) \setminus \overline{B}(y,R_m)$.
By Theorem~\ref{thm:main_thm_propagation_of_index_bounds}\ref{item:main_thm_ind_reg}
\begin{equation}
\ind_{B(y,R) \setminus \overline{B}(y,R_m)} B_V = 0
\text{ for all $R_m \to 0$,}
\end{equation}
and thus $\ind_{B(y,R) \setminus \{y\}} B_V = 0$.

By contradiction, suppose that $y \in \sing V$. Then $\ind_{B(y,r)} B_V = 0$
holds in the whole ball $B(y,r)$ away from $\sing V$, and
the regularity results of~\cite{Wickramasekera14} give
$\dimh B(y,R) \cap \sing V \leq n-7$, so that $y \notin \sing V$.
\end{proof}
Claim~\ref{claim:y_not_in_sing_V} concludes the induction step;
together with the base of induction, we have proved that
$\dimh \sing V \leq n-7$. This finishes the proof of
Theorem~\ref{thm:main_thm_propagation_of_index_bounds}.

\appendix

\section{Measure-function convergence}
\label{app:measure_function_convergence}

In this appendix we give two elementary measure-theoretical lemmas
that are used in the proofs of
Lemma~\ref{lem:min_max_convergence_when_C_c_convergence} and in
Appendix~\ref{app:second_fundamental_form_coordinate_expression}.
Essentially they give a weak compactness result for sequences of the form
$(f_i \dmui)_{i \in \NN }$, with $\mu_i$ Radon measures
and $f_i \in L^2(\mu_i)$.
The weak convergence in question is sometimes called
\emph{measure-function convergence} in the literature.
It appears in the work of Hutchinson~\cite{Hutchinson86} on
so-called \emph{curvature varifolds}; there one also finds a proof of
Lemma~\ref{lem:measure_function_convergence} under more general hypotheses
on the sequence $(f_i)$.

\begin{lem}[\cite{Hutchinson86,Tonegawa05}]
	\label{lem:measure_function_convergence}
	Let $X$ be a locally compact Hausdorff space, let $(\mu_i)_{i \in \NN}$ be
	a sequence of Radon measures on $X$, and $(f_i: X \to \RR)_{i \in \NN}$ be a sequence
	of Borel-measurable functions. Suppose that
	\begin{align}
	\sup_i \mu_i(X) &< +\infty, \\
	\sup_i \int_X  f_i ^2 \intdiff \mu_i &< +\infty.
	\end{align}
	Then there is a Radon measure $\mu$ and $f \in L^2(\mu)$
	such that for some subsequence
	$\mu_{i'} \rightharpoonup \mu$ and $f_{i'} \intdiff \mu_{i'}
	\rightharpoonup f \intdiff \mu$ weakly as Radon measures, i.e.\
	\begin{equation}
	\int_X  f_{i'} \phi \intdiff \mu_{i'}
	\to
	\int_X  f  \phi  \intdiff \mu
	\quad \text{for all} \; \phi \in C_c(X).
	\end{equation}
	Moreover, the weak limit $f \intdiff \mu$ satisfies
	\begin{equation}
	\label{eq:measure_function_convergence_L2_bound}
	\int_X f^2 \intdiff \mu
	\leq \liminf_{i \to \infty} \int_X f_i^2 \intdiff \mu_i. 
	\end{equation}
\end{lem}
\begin{rem}
	In our applications $X$ is either an open subset of $U \subset M$
	or its Grassmannian $G_n(U)$, and $\mu_i$ is either $\nvi$ or $\vi$.
\end{rem}

\begin{proof}
	The signed Radon measures $\nu_i := f_i \dmui$ have bounded total variation,
	so that the sequences $(\mu_i)$ and $(\nu_i)$ have convergent subsequences,
	with limits the Radon measures $\mu$ and $\nu$ respectively. We extract
	these subsequences without relabelling their indices.
	
	Consider an arbitrary $\phi \in C_c(X)$.
	By the weak convergence $\nu_i \rightharpoonup \nu$,
	\begin{equation}
	\int \phi \intdiff \nu = \lim_{i \to \infty} \int \phi f_i \dmui
	\leq \norm{\phi}_{L^2(\mu)}
	\liminf_{i \to \infty} \norm{f_i}_{L^2(\mu_i)},
	\end{equation}
	where we used the weak convergence $\mu_i \rightharpoonup \mu$ to get
	$\lim_{i \to \infty} \norm{\phi}_{L^2(\mu_i)} = \norm{\phi}_{L^2(\mu)}$.
	As $C_c(X)$ is dense in $L^2(\mu)$, the measure $\nu$ defines a
	bounded linear functional on $L^2(\mu)$, and by duality there is
	$f \in L^2(\mu)$ with 
	$\norm{f}_{L^2(\mu)} \leq \liminf_{i \to \infty} \norm{f_i}_{L^2(\mu_i)}$ such that $\nu = f \dmu$.
\end{proof}
If the densities $f_i$ are in $C_c(X)$ and converge strongly, then
their limit coincides with the density of the weak limit of
$(f_i \dmui)$.
\begin{cor}
	\label{cor:meas_fct_cvg}
	Additionally to the hypotheses of Lemma~\ref{lem:measure_function_convergence},
	assume that $f_i \in C_c(X)$,
	and that $\norm{f_i - f}_{L^\infty} \to 0$ for some $f \in C_c(X)$.
	Then, additionally
	to the conclusions of Lemma~\ref{lem:measure_function_convergence},
	\begin{equation}
	\label{eq:meas_fct_cvg_limit}
	\int_X f^2 \dmu = \lim_{i \to \infty} \int_X f_i^2 \dmui.
	\end{equation}
\end{cor} 
\begin{proof}
	We first show that $f_i \intdiff \mu_i \rightharpoonup f \intdiff \mu$.
	Let $\varphi \in C_c(X)$ be arbitrary, then
	\begin{equation}
	\int f_i \varphi \dmui - \int f \varphi \dmu
	= \int(f_i - f) \varphi \dmui
	+ \int f \varphi \dmui - \int f  \varphi \dmu.
	\end{equation}
	The first term $\abs{\int(f_i - f) \varphi \dmui}
	\leq \norm{f_i - f}_{L^\infty} \norm{\varphi}_{L^1(\mu_i)} \to 0$
	as $i \to \infty$.
	The remaining terms $\int f \varphi \dmui - \int f \varphi \dmu
	\to 0$ as $i \to \infty$ by the weak convergence $\mu_i \rightharpoonup \mu$.
	We reason similarly to show \eqref{eq:meas_fct_cvg_limit}:
	\begin{equation}
	\left| \int_X f^2 \dmu - \int_X f_i^2 \dmui \right|
	\leq \left| \int_X f^2 \dmu - \int_{X} f^2 \dmui\right|
	+ \mu_i(X) \norm{f_i^2 - f^2}_{L^\infty} .
	\end{equation}
	The first term goes to $0$ by the weak convergence $\mu_i \rightharpoonup \mu$,
	and so does the second as $\sup_i \mu_i(X) < +\infty$ and
	$\norm{f_i^2 - f^2}_{L^\infty} \to 0$.
\end{proof}

\section{Generalised second fundamental forms}
\label{app:second_fundamental_form_coordinate_expression}

Our main aim in this appendix is to give a proof of Proposition~\ref{prop:weak_convergence_second_ff}.
We follow the approach of~\cite{Tonegawa05}, where the case of stable $u_i$ is treated using notions from~\cite{Hutchinson86}.
Our account is self-contained but for the fact that we refer to
these two works for some technical, but routine calculations.

Throughout this section, we assume that $U \subset M$ is isometrically
embedded in some $\RR^q$, and $W \subset \subset U \setminus \sing V$ is
an open subset with $W \cap \reg V \neq \emptyset$.
The fibre of the Grassmannian $G_n(U)$ at $x \in U$ is
identified with the subspaces
\begin{equation}
\{S \subset \RR^q \mid S \subset T_x M, \dim S = n \}
\subset  U \times G(n,q), \end{equation}
where $G(n,q) = \{S \subset \RR^q \mid \dim S = n\}$.
We furthermore identify an element $S \in U \times G(n,q)$
with the corresponding orthogonal projection $\RR^q \to S$,
so that $G_n(U) \subset U \times \RR^{q^2}$.
Throughout, $P(x) \in \RR^{q^2}$ represents the orthogonal projection
$\RR^q \to T_x M$ and $(e_1,\dots,e_q)$ is the standard basis of $\RR^q$;
$\partial_i$ and $\partial^*_{ij}$ denote differentiation with respect to $e_i$ and
$e_i \otimes e_j \in \RR^{q^2}$ respectively.

Consider first a smooth embedded hypersurface $\Sigma \subset U$,
which we implicitely identify with the varifold $V_\Sigma := V_{\Sigma,1}$ with constant multiplicity.
We consider test functions $\phi \in C^1(U \times \RR^{q^2})$
with compact spatial support, that is with compact support in the
first variable.
We associate to it a function $\varphi \in C_c^1(\Sigma)$ defined by
$\varphi(x) = \phi(x,S^\Sigma(x))$,
where $S^\Sigma(x) \in \RR^{q^2}$ is the orthogonal projection
$\RR^q \to T_x \Sigma$.
Define a vector field $X \in C_c^1(\Sigma,TM)$ by $X = \varphi P(e_j)$, where $e_j$ is one of the standard basis vectors.
Its component tangential to $\Sigma$ is
$\varphi S^\Sigma(e_j)$, and by the standard divergence theorem
we get $\int_\Sigma \div_\Sigma (\varphi S^\Sigma e_j) = 0$.
A routine calculation shows that in coordinates
\begin{equation}
\label{eq:div_coord_express}
\div_\Sigma (\varphi S^\Sigma e_j)
= S^\Sigma_{rj} \pd_r \phi + \phi S^\Sigma_{ri} \pd_i S^\Sigma_{jr}
+  S^\Sigma_{ji} \pd_i S^\Sigma_{kr} \pd^*_{kr} \phi,
\end{equation}
with summation over repeated indices~\cite{Hutchinson86}.
Abbreviate $B^\Sigma_{jkr} = S^\Sigma_{ji} \pd_i S^\Sigma_{kr}$
and substitute this into the divergence formula:
\begin{align}
0 &= \int_\Sigma S^\Sigma_{rj} \pd_r \phi
+ B_{rjr}^\Sigma \phi
+ B_{jkr}^\Sigma \pd^*_{kr} \phi \dH^n \\
&= \int_{G_n(U)} S_{rj} \pd_r \phi
+ B_{rjr}^\Sigma \phi
+ B_{jkr}^\Sigma \pd^*_{kr} \phi \intdiff V_\Sigma(x,S). \label{eq:gen_sec_ff_def_id}
\end{align}
This identity is the basis of the following definition
by Hutchinson~\cite{Hutchinson86}.
\begin{defn}[Generalised curvature, \cite{Hutchinson86}]
\label{defn:generalised_curvature}
An $n$-dimensional integral varifold $V$ in $U$
is said to have \emph{generalised curvature} if there exists
a function $B = (B_{ijk})$ with values in
$\RR^{q^3}$ defined $V$-a.e.\ on $G_n(U)$ with
\begin{enumerate}[label = (\alph{*})]
	\item $B \in L^1_{\mathrm{loc}}(V)$, \label{item:gen_curv_a}
	\item \label{item:gen_curv_b}
	$\int_{G_n(U)} S_{rj} \pd_r \phi 
	+ B_{rjr} \phi
	+ B_{jkr} \pd^*_{kr} \phi \intdiff V(x,S) = 0$ for
	all $\phi \in C^1(U \times \RR^{q^2})$ with compact spatial support.
\end{enumerate}
\end{defn}
The following lemma shows that the function $B$ is well-defined $V$-a.e.\ on $G_n(U)$;
it is taken from~\cite{Hutchinson86}.
\begin{lem}[\cite{Hutchinson86}]
	Any two $B$ and $\widetilde{B}$
	satisfying \ref{item:gen_curv_a}~and~\ref{item:gen_curv_b} coincide
	$V$-a.e.\  on $G_n(U)$.
\end{lem}
\begin{proof}
	Let $\phi (x,S) = \alpha(x) \beta(S)$, where $\alpha \in C_c^1(U)$ and $\beta \in C^1(\RR^{q^2})$.
	Letting $\beta \equiv 1$ we see that
	$
	\int B_{rjr} \alpha \dv = \int \widetilde{B}_{rjr} \alpha \dv$, and thus
	$B_{rjr} = \widetilde{B}_{rjr}$ $V$-a.e.\ on $G_n(U)$. If we now let
	\begin{equation}
	\beta(S) = \begin{cases}
	1  &\text{if } S = S_{kr}\\
	0 &\text{otherwise}
	\end{cases}
	\end{equation}
	then $\int B_{jkr} \alpha \dv = \int \widetilde{B}_{jkr} \alpha \dv$,
	whence the conclusion follows.
\end{proof}
In particular, applied to the smooth hypersurfaces, the following is an immediate consequence.
\begin{cor}
	\label{cor:uniqueness_gen_curv}
	If $\Sigma$ is a smoothly embedded hypersurface, then $B_{ijk}(x,S) = B_{ijk}^\Sigma(x,S)$
	for $V_\Sigma$-a.e. $(x,S) \in G_n(U)$, where $B_{ijk}^\Sigma(x,S) = S_{il} \pd_l S^\Sigma_{jk}$.
\end{cor}
The following
elementary calculation relates $B^\Sigma$ to the second fundamental form $A^\Sigma$.
\begin{lem}
	\label{lem:sec_form_coordinate_expression}
	Let $A^\Sigma$ be the second fundamental form of a smoothly embedded
	hypersurface $\Sigma \subset U$. Then
	\begin{equation}
	\label{eq:second_ffs_coincide}
	\langle A^\Sigma(S^\Sigma e_i,S^\Sigma e_j),Pe_k \rangle
	=  P_{kr} S^\Sigma_{js} S^\Sigma_{il} \pd_l S^\Sigma_{rs}
	= P_{kr}  S^\Sigma_{js} B_{irs}^\Sigma(x,S^\Sigma).
	\end{equation} 
\end{lem}
\begin{proof}
Write $A$ instead of $A^\Sigma$ in this proof to simplify notation.
The covariant derivative on $M$ is the component of
$D = \nabla^{\RR^q}$ tangent to $M$, so
$A = (D^{T_M})^{\perp_\Sigma} = (D^{\perp_\Sigma})^{T_M}$.
As $e_k^{T_M} = P_{kr} e_r$, we get
\begin{align}
\begin{split}
A^k_{ij} &:= \langle A(S^\Sigma e_i, S^\Sigma e_j), e_k^{T_M} \rangle \\
&= \langle (D_{S^\Sigma e_i} S^\Sigma e_j)^{\perp_\Sigma},P e_k \rangle 
= P_{kr} \langle D_{S^\Sigma e_i} S^\Sigma e_j,e_r^{\perp_\Sigma} \rangle.
\end{split}
\end{align}
Similarly $e_r^{\perp_\Sigma} = (\delta_{rs} - S^\Sigma_{rs})e_s$, so
\begin{equation}
\label{eq:appendix_intermediary_identity_for_A}
A_{ij}^k = P_{kr}(\delta_{rs} - S^\Sigma_{rs}) \langle D_{S^\Sigma e_i} S^\Sigma e_j,e_s \rangle
= P_{kr} (\delta_{rs} - S^\Sigma_{rs}) D_{S^\Sigma e_i} S^\Sigma_{js}.
\end{equation}
As $S^\Sigma_{rs} S^\Sigma_{js} = S^\Sigma_{rj}$, we finally get
$A_{ij}^k = P_{kr} S^\Sigma_{js} D_{S^\Sigma e_i} S^\Sigma_{rs} = P_{kr} S^\Sigma_{js} S^\Sigma_{il} \pd_l S^\Sigma_{rs}$,
as required.
\end{proof}
This expression can then be used to generalise second fundamental forms
from the smooth to the varifold setting.
\begin{defn}[Generalised second fundamental forms, \cite{Hutchinson86}]
	Let $V$ be an integral $n$-varifold with generalised curvature.
	Then its \emph{generalised second fundamental form} is the function
	$A = (A_{ij}^k)$ with values in $\RR^{q^3}$ and defined
	at $V$-a.e.\ $(x,S) \in G_n(U)$ by
	\begin{equation}
	A_{ij}^k(x,S) = P_{kr} S_{js} B_{irs}.
	\end{equation}
\end{defn}

For a smoothly embedded $\Sigma \subset U$, we see after
combining Corollary~\ref{cor:uniqueness_gen_curv} with
Lemma~\ref{lem:sec_form_coordinate_expression} that the generalised
second fundamental form $A$ of $V_\Sigma$ is equal to the classical second
fundamental form $A^\Sigma$ in the sense that
\begin{equation}
A_{ij}^k(x,S) = \langle A^\Sigma(S e_i,S e_j),Pe_k \rangle
\quad \text{for $V_\Sigma$-a.e.\ $(x,S) \in G_n(U)$}.
\end{equation}

We now want to relate these notions to the varifolds $V^i$ defined
in the main body. To simplify notation, we fix a critical point $u = u_i$
with associated varifold $V^\eps = V^i$.
We define a `second fundamental form' for $V^\eps$ using the
coordinate expressions from Lemma~\ref{lem:sec_form_coordinate_expression}.
\begin{defn}
	\label{defn:second_ff_app_grass}
	Define the functions $A^\eps = (A^{\eps,k}_{ij})$ and $B^\eps = (B^\eps_{ijk})$
	with values in $\RR^{q^3}$ at all $(x,S) \in G_n(U)$ where
	$\nabla u(x) \neq 0$ by
	\begin{align}
	B^\eps_{ijk}(x,S) &= S_{il} \pd_l S_{jk}^\eps, \label{eq:be_defn} \\
	A^{k,\eps}_{ij}(x,S) &= P_{kr} S_{js} S_{il} \pd_l S_{rs}^\eps = P_{kr} S_{js} B^\eps_{irs},
	\end{align}
	where $S^\eps = S^\eps(x) \in \RR^{q^2}$ represents the projection
	$\RR^q \to T_x \{ u  =  u(x) \}$, and $P = P(x) \in \RR^{q^2}$
	the projection $\RR^q \to T_x M$.
\end{defn}
Technically speaking the function $A^\eps$ is not the second fundamental
form of $V^\eps$, as $B^\eps$ satisfies the integral identity of
Definition~\ref{defn:generalised_curvature} only up to a small error term.
This can be seen as follows:
take an arbitrary vector field $X \in C_c^1(U,TM)$,
multiply the Allen--Cahn equation by $\langle X , \nabla u \rangle$
and integrate by parts twice to obtain
\begin{equation}
\int_U \abs{\nabla u}^2 \div X - \langle \nabla_{\nabla u} X, \nabla u \rangle
= \int_U \left(\frac{\abs{\nabla u}^2}{2} - \frac{W(u)}{\eps^2}\right) \div X,
\end{equation}
which using integration with respect to $V^\eps$ is equivalent to
\begin{equation}
\label{eq:guaraco_identity}
\int_{G_n(U)} \div_S X \dve(x,S)
= 	\frac{1}{2 \sigma} \int_U 	\left( \epsilon \frac{\lvert \nabla u \rvert^2}{2}
- \frac{W(u)}{\epsilon} \right)
\div X,
\end{equation}
where $\div_S X = \sum_{i=1}^{q} \langle D_{Se_i} X, Se_i \rangle$.
As before let $X = \phi(x,S^\eps) S^\eps(e_j)$, where
$\phi \in C^1(U \times \RR^{q^2})$ has compact spatial support.
Substitute this into \eqref{eq:guaraco_identity}
and perform routine coordinate computations as before in
\eqref{eq:div_coord_express} to get
\begin{multline}
\label{eq:approximate_identity_second_fundamental_form}
\int_{G_n(U)} S_{rj} \pd_r \phi
+ B_{rjr}^\epsilon \phi
+ B_{jkr}^\epsilon \pd^*_{kr} \phi \dve(x,S)
= \\ \frac{1}{2 \sigma} \int_U
\left( \epsilon \frac{\lvert \nabla u \rvert^2}{2}
- \frac{W(u)}{\epsilon} \right) \div X.
\end{multline}
The integral on the right-hand side goes to $0$ as $\eps \to 0$---%
this is \eqref{eq:thm_properties_V_equipartition_energy}
in Theorem~\ref{lem:properties_V}.
This justifies the abuse of language that is calling $A^\eps$ the `second fundamental
form' of $V^\eps$.

We now compare $A^{\eps}$ with the second fundamental form $A^{\Sigma}$
of the level sets of $u$ near regular points: take a point $x \in U$ with
$\nabla u(x) \neq 0$.
Then the level set $\{u = u(x)\}$ is embedded in a neighbourhood
$B$ around $x$; write $\Sigma = \{ u = u(x)\} \cap B$.
The calculations from Lemma~\ref{lem:sec_form_coordinate_expression}
show that $A^\eps(x,S^\eps) = A^\Sigma(x)$, so
the second fundamental forms from Definitions~\ref{defn:second_ff}
and~\ref{defn:second_ff_app_grass} agree $V^\eps$-a.e.
Combining this observation with \eqref{eq:ae_bounded_hessian}, we get
\begin{equation}
\label{eq:norm_second_ff_and_ae}
\abs{\Ae}^2(x,S^\eps)
= \abs{A^\Sigma}^2(x)
\leq \frac{1}{\abs{\nabla u}^2}
(\abs{\nabla^2 u}^2 - \abs{\nabla \abs{\nabla u}}^2).
\end{equation}
Therefore, when
$\delta^2 E_{\eps}(u)(\abs{\nabla u}\phi,\abs{\nabla u}\phi) \geq 0$
for some $\phi \in C_c^1(U)$, then 
$\int_{G_n(U)} \lvert A^\epsilon \rvert^2 \phi^2 \dve
\leq \int_{U} \lvert \nabla \phi \rvert^2 
- \Ric(\nu^\epsilon,\nu^\epsilon) \phi^2 \intdiff \norm{\ve}$
as in \eqref{eq:second_inequality_stability}, and Corollary~\ref{cor:L2_bound_for_Ae_and_Be_if_u_stable}
also remains valid.

All the results in this appendix were established for an arbitrary
critical point $u \in C^3(U) \cap L^\infty(U)$, and thus are
valid for every term in the sequence $(u_i)$ satisfying Hypotheses
(A)--(C). Let $(V^i)$ be the corresponding varifolds
as in \eqref{eq:definition_varifolds}, and
let $(A^{\eps_i})$ be their second fundamental forms as
in Definition~\ref{defn:second_ff_app_grass}.
We restate Proposition~\ref{prop:weak_convergence_second_ff} in the
following equivalent form, with $A^{\eps_i}$ in place of $A^i$.
\begin{prop}
	\label{prop:app:weak_convergence_second_ff}
	If $\sup_i \int_W \abs{A^{\eps_i}}^2 \dnvi < +\infty$,
	then some subsequence $A^{\eps_{i'}} \dveip \rightharpoonup A \dv$
	weakly as Radon measures on $G_n(W)$, and
	\begin{equation}
	\label{eq:sec_ff_inequality}
	\int_W \abs{A}^2 \dnv \leq
	\liminf_{i \to \infty} \int_{W} \abs{A^{\eps_i}}^2 \dnvi,
	\end{equation}
	where $A$ is the classical second fundamental form of $\reg V \subset M$.
\end{prop}
\begin{proof}
	Routine calculations as above in the proof of 
	Lemma~\ref{lem:sec_form_coordinate_expression} show that
	$\Bei$ is related to $A^{\eps_i}$ as follows for all $i$:
	\begin{equation}
	B_{jkl}^{\eps_i}(x,S^{\eps_i}) = A_{jk}^{l,\eps_i} + A_{jl}^{k,\eps_i} +
	S_{ks}^{\eps_i} S^{\eps_i}_{jr} \pd_r P_{sl} +
	S_{ls}^{\eps_i} S_{jr}^{\eps_i} \pd_r P_{ks}.
	\label{eq:second_identity_auxiliary_quantities}
	\end{equation}
	If we square \eqref{eq:second_identity_auxiliary_quantities} and sum
	over $j,k,l=1,\dots,q$, we get
	\begin{equation}
	\label{eq:almost_everywhere_bound_for_auxiliary_quantity}
	\abs{\Bei}^2 \leq 8 (\abs{A^{\eps_i}}^2 + \abs{DP}^2)
	\quad \text{$\vei$-a.e. in } G_n(U)
	\end{equation}
	The term $\abs{DP}^2 := \sum_{j,k,l}^q (\pd_j P_{kl})^2$ can be bounded
	by some constant $C(M)$, so that $\sup_i \int_W \abs{B^{\eps_i}}^2 \dnvi < +\infty$
	as well.
	
	By Lemma~\ref{lem:measure_function_convergence} we can pass to
	convergent subsequences $A^{\eps_i} \dvi \rightharpoonup A \dv$ and
	$B^{\eps_i} \dvi \rightharpoonup B \dv$ with limits related by
	$A_{jk}^l = P_{lr} S_{ks} B_{jrs}$ $V$-a.e. in $G_n(W)$.
	The limit $A \dv$ also satisfies
	\begin{equation}
	\int_{G_n(W)} \abs{A}^2 \dv \leq \liminf_{i \to \infty} \int_{G_n(W)}
	\abs{A^{\eps_i}}^2 \dvi,
	\end{equation}
	and the analogous inequality holds for $B \dv$.
	Moreover the error term on the right-hand side of
	\eqref{eq:approximate_identity_second_fundamental_form}
	tends to $0$ as $i \to \infty$, so the weak limit $B \dv$ satisfies
	\begin{equation}
	\label{eq:gen_second_ff}
	\int_{G_n(W)} S_{rj} \pd_r \phi
	+ B_{rjr} \phi
	+ B_{jkr} \pd^*_{kr} \phi \dv(x,S)
	= 0
	\end{equation}
	for all $\phi \in C^1(W \times \RR^{q^2})$
	with compact spatial support.
	By Corollary~\ref{cor:uniqueness_gen_curv} we have
	$B = B^{\reg V}$ and thus also
	$A = A^{\reg V}$ $V$-a.e.\ in $G_n(W)$. This concludes
	the proof.
\end{proof}
%

\emergencystretch=2.5em

\printbibliography



\end{document}